\documentclass{amsart}

\usepackage{fullpage}

\usepackage{amssymb}
\usepackage{amsfonts}
\usepackage{amsmath}
\usepackage{amsthm}
\usepackage{bm}
\usepackage{color}      
\usepackage{comment}
\usepackage{enumerate}
\usepackage{graphicx}
\usepackage{esint}
\usepackage{subfig}
\usepackage{url}
\usepackage{mathtools}
\mathtoolsset{showonlyrefs} 

\newtheorem{thm}{Theorem}

\newtheorem{prop}[thm]{Proposition}

\newtheorem{remark}[thm]{Remark}

\newcommand{\tr}{\mathrm{tr}}        
\newcommand{\rank}{\mathrm{rank}}        
\newcommand{\sign}{\mathrm{sgn}}        
\newcommand{\dd}{:} 
\newcommand{\tp}{^{T}} 
\newcommand{\dij}{\delta_{ij}}

\newcommand{\vu}{\mathbf{u}}
\newcommand{\vv}{\mathbf{v}}
\newcommand{\vw}{\mathbf{w}}

\newcommand{\vQ}{\mathbf{Q}}

\newcommand{\vm}{\mathbf{m}}
\newcommand{\vn}{\mathbf{n}}
\newcommand{\vt}{\mathbf{t}}

\newcommand{\vx}{\mathbf{x}}
\newcommand{\vy}{\mathbf{y}}
\newcommand{\vnu}{\bm{\nu}}

\newcommand{\vPi}{\mathbf{\Pi}}

\newcommand{\vA}{\mathbf{A}}
\newcommand{\vB}{\mathbf{B}}

\newcommand{\vU}{\mathbf{U}}

\newcommand{\vW}{\mathbf{W}}
\newcommand{\vV}{\mathbf{V}}
\newcommand{\vwU}{\widetilde{\mathbf{U}}}

\newcommand{\vM}{\mathbf{M}}
\newcommand{\vN}{\mathbf{N}}
\newcommand{\vP}{\mathbf{P}}

\newcommand{\vI}{\mathbf{I}}
\newcommand{\vR}{\mathbf{R}}
\newcommand{\vT}{\mathbf{T}}
\newcommand{\vF}{\mathbf{F}}

\newcommand{\vzero}{\mathbf{0}}

\newcommand{\Om}{\Omega}
\newcommand{\dOm}{\partial \Omega}
\newcommand{\Gm}{\Gamma}

\newcommand{\iO}{\int_{\Omega}}

\newcommand{\X}{\mathbb{X}}
\newcommand{\Sh}{\mathbb{S}_h}   
\newcommand{\Nh}{\mathbb{N}_h}   
\newcommand{\THh}{\mathbb{T}_h}   
\newcommand{\Uh}{\mathbb{U}_h}   
\newcommand{\Sp}{\mathbb{S}}  
\newcommand{\V}{\mathbb{N}^{\perp}}   
\newcommand{\Vh}{\V_h}   
\newcommand{\LL}{\mathbb{L}}  
\newcommand{\R}{\mathbb{R}}   

\newcommand{\Sing}{\mathbb{S}}   

\newcommand{\eps}{\varepsilon}

\newcommand{\Tk}{\mathcal{T}}

\newcommand{\Nk}{\mathcal{N}}

\newcommand{\ELdG}{E_{\mathrm{LdG}}}
\newcommand{\ELdGone}{E_{\mathrm{LdG},\mathrm{one}}}
\newcommand{\ELdGfunc}{\mathcal{W}_{\mathrm{LdG}}}

\newcommand{\Li}{L}
\newcommand{\Ebulk}{E_{\mathrm{LdG},\mathrm{bulk}}}
\newcommand{\Bulkfunc}{\phi_{\mathrm{LdG}}}
\newcommand{\BulkLdG}{\psi_{\mathrm{LdG}}}

\newcommand{\Bulkcoef}{\eta_{\mathrm{B}}}
\newcommand{\BulkA}{A}
\newcommand{\BulkB}{B}
\newcommand{\BulkC}{C}

\newcommand{\BulkK}{K}

\newcommand{\consistLdG}{\mathcal{E}} 
\newcommand{\consistULdG}{\widetilde{\mathcal{E}}} 

\newcommand{\vNN}{{\bf\Theta}}
\newcommand{\vtNN}{\widetilde{\bf\Theta}}
\newcommand{\vttNN}{\widehat{\bf\Theta}}   

\newcommand{\Euni}{E_{\mathrm{uni}}}

\newcommand{\Euniring}{{E}_{\mathrm{uni-i}}} 
\newcommand{\Eunis}{{E}_{\mathrm{uni-s}}} 
\newcommand{\Eunit}{{E}_{\mathrm{uni-t}}} 
\newcommand{\Admisuni}{\mathcal{A}_{\mathrm{uni}}}   
\newcommand{\Eunimain}{E_{\mathrm{uni}-\mathrm{m}}}
\newcommand{\EuniUmain}{\widetilde{E}_{\mathrm{uni}-\mathrm{m}}}
\newcommand{\EuniUtotal}{\widetilde{E}_{\mathrm{uni}-\mathrm{t}}}

\newcommand{\Eerkmain}{E_{\mathrm{erk}-\mathrm{m}}}
\newcommand{\EerkUmain}{\widetilde{E}_{\mathrm{erk}-\mathrm{m}}}
\newcommand{\Eerkone}{E_{\mathrm{erk}}}

\newcommand{\DWfunc}{\phi_{\mathrm{erk}}}

\newtheorem{theorem}{Theorem}
\newtheorem{lemma}{Lemma}

\begin{document}

\title[Structure-preserving FEM for uniaxial liquid crystals]{A structure-preserving FEM for the uniaxially constrained $\vQ$-tensor model of nematic liquid crystals}

\author[J.P.~Borthagaray]{Juan Pablo~Borthagaray}
\address[J.P.~Borthagaray]{Department of Mathematics, University of Maryland, College Park, MD 20742, USA, and Departamento de Matem\'atica y Estad\'istica del Litoral, Universidad de la Rep\'ublica, Salto, Uruguay}
\email{jpborthagaray@unorte.edu.uy}
\thanks{JPB has been supported in part by NSF grant DMS-1411808}

\author[R.H.~Nochetto]{Ricardo H.~Nochetto}
\address[R.H.~Nochetto]{Department of Mathematics and Institute for Physical Science and Technology, University of Maryland, College Park, MD 20742, USA}
\email{rhn@math.umd.edu}
\thanks{RHN has been supported in part by NSF grants DMS-1411808 and DMS-1908267}

\author[Shawn W.~Walker]{Shawn W.~Walker}
\address[Shawn W.~Walker]{Department of Mathematics and Center for Computation and Technology (CCT) Louisiana State University, Baton Rouge, LA 70803}
\email{walker@math.lsu.edu}
\thanks{SWW has been supported in part by NSF grant DMS-1555222 (CAREER)}


\begin{abstract}
We consider the one-constant Landau - de Gennes model for nematic liquid crystals. The order parameter is a traceless tensor field $\vQ$, which is constrained to be uniaxial: $\vQ = s (\vn\otimes\vn - d^{-1}\vI)$ where $\vn$ is a director field, $s\in\R$ is the degree of orientation, and $d\ge2$ is the dimension. Building on similarities with the one-constant Ericksen energy, we propose a structure-preserving finite element method for the computation of equilibrium configurations. We prove stability and consistency of the method without regularization, and $\Gamma$-convergence of the discrete energies towards the continuous one as the mesh size goes to zero. We design an alternating direction gradient flow algorithm for the solution of the discrete problems, and we show that such a scheme decreases the energy monotonically. Finally, we illustrate the method's capabilities by presenting some numerical simulations in two and three dimensions including non-orientable line fields.
\end{abstract}

\maketitle

\section{Introduction}\label{sec:intro}

The liquid crystal state of matter is observed in certain materials as a mesophase between the crystalline and the isotropic liquid phases. Such a state may be obtained as a function of temperature between the two latter phases; in this case, these are called thermotropic liquid crystals. Other classes include lyotropic and metallotropic liquid crystals, in which concentration of the liquid-crystal molecules in a solvent or the ratio between organic and inorganic molecules determine the phase transitions, respectively. In this paper, we consider thermotropic liquid crystals \cite{deGennes_book1995}.

The physical state of a material can be described in terms of the translational and rotational motion of its constituent molecules. In a crystalline solid, molecules exhibit both long-range ordering of the positions of the centers and orientation of the molecules. As the substance is heated, the molecules gain kinetic energy and large molecular vibrations usually make these two ordering types disappear at the same temperature. This results in a fluid phase. However, in some materials, that typically consist of either rod-like or disc-like molecules, the long-range orientational ordering survives until a higher temperature than the long-range positional ordering. Such a state of matter is called {\em liquid crystalline}. Moreover, when long-range positional ordering is completely absent, the liquid crystal is regarded as {\em nematic}. 

On average, nematic liquid crystal molecules are aligned with their long axes parallel to each other. At the macroscopic level, this means that there is a preferred direction; often, such a direction is a rotational symmetry axis. In such a case, the nematic liquid crystal phase is {\em uniaxial}. If, in contrast, there is no such rotational symmetry, then the material is in a {\em biaxial} state. 

Depending on the choice of {\em order parameter} (cf. Section \ref{sec:order_param}), several models for nematic liquid crystals have been proposed. Because the vast majority of thermotropic liquid crystals exhibit uniaxial behavior, this is often built into the modeling. If one takes as order parameter the orientation of the molecules $\vn(x) \in \Sp^{d-1}$, for $x\in\Omega\subset\R^d$, then $\vn$ is a harmonic mapping in the domain $\Omega$; numerical methods for this model have been proposed, for example, in \cite{Adler_SJNA2015, Alouges_SJNA1997, Bartels_MC2010b, Cohen_CPC1989, Gartland_SJNA2015, Lin_SJNA1989}. We refer also to \cite{Cruz_JCP2013, Guillen-Gonzalez_M2AN2013,Liu_SJNA2000,Walkington_M2AN2011} for discretizations of liquid crystal flows. It is often the case that liquid crystal configurations display {\em defects}, that is, that the molecular orientation is not continuous in some regions of the material. Harmonic map models do not allow for point defects if $d = 2$ or line defects if $d = 3$, because the energy is singular.

However, if besides the liquid crystal molecule orientation $\vn$ one considers a scalar variable $s(x)$ that represents the degree of alignment that molecules have with respect to $\vn$, then the equilibrium configuration minimizes the Ericksen energy \cite{deGennes_book1995,Ericksen_ARMA1991,Virga_book1994}. Minimizers of such an energy can exhibit nontrivial defects, as the parameter $s$ can relax a large contribution from $|\nabla \vn|$, and wherever the degree of alignment $s$ vanishes, the resulting Euler-Lagrange equation for $\vn$ is degenerate. Finite element methods for the Ericksen model have been used to approximate both equilibrium configurations \cite{NochettoWalker_MRSproc2015,Nochetto_SJNA2017,Nochetto_JCP2018} and dynamics \cite{Barrett_M2AN2006} of the molecular orientation.

If one considers the probability distribution of the liquid crystal molecules orientation and chooses to use its second moments to define an order parameter, then this leads to the Landau - de Gennes model. In such a model, the order parameter is a tensor field $\vQ (x)$ that measures the discrepancy between the probability distribution at $x\in\Omega$ and a uniform distribution on $\Sp^{d-1}$. Numerical methods for the Landau - de Gennes energy are considered in \cite{Bajc_JCP2016,Bartels_bookch2014,Davis_SJNA1998,Gartland_MCLC1991,James_IEEE2006,Schopohl_PRL1987}.

In this work, we shall be concerned with uniaxial nematic liquid crystals in $\R^d$ for $d \ge 2$; we present numerical experiments for $d = 2, 3$. Our goal is to design a finite element method for a uniaxially-constrained $\vQ$-tensor model, and to prove stability and convergence properties.
More precisely, we prove that if the corresponding meshes are weakly acute, then our discrete energy $\Gamma$-converges to the continuous one as the mesh size goes to zero. Our method can handle the degeneracy introduced by a vanishing degree of orientation without any regularization. Moreover, because the $\vQ$-tensor approach incorporates a head-to-tail symmetry into the modeling, our approach is able to capture \emph{non-orientable} equilibrium configurations. 

The paper is organized as follows. In Section \ref{sec:modeling_LCs}, we discuss modeling of the equilibrium states of liquid crystals. We examine the Landau - de Gennes and Ericksen energies, and discuss the capabilities of these models to capture defects. Section \ref{sec:formulation} is devoted to the formulation of the problem we study in this paper. Such a section includes a discussion on previous work for the Ericksen model \cite{Nochetto_SJNA2017}, which is instrumental for our numerical method. We introduce the discrete setting for the uniaxially-constrained Landau - de Gennes energy and prove key energy inequalities in Section \ref{sec:LdG_discretization}. Afterwards, in Section \ref{sec:LdG_Gamma_conv}, we prove the $\Gamma$-convergence of the discrete energies. For the computation of discrete minimizers, in Section \ref{sec:LdG_Grad_Flow} we propose a gradient flow and prove a strictly monotone energy decreasing property. Finally, Section \ref{sec:LdG_experiments} presents numerical experiments for $d=2,3$ illustrating the capabilities of our method.

\section{Modeling of nematic liquid crystals}\label{sec:modeling_LCs}

We discuss some elementary properties of the so-called $\vQ$-tensors and review three models for the equilibrium states of nematic liquid crystals, which derive from minimizing an energy (see \cite{deGennes_book1995,Virga_book1994,Mottram_arXiv2014} for more details on the modeling of liquid crystals).

\subsection{Order Parameters}\label{sec:order_param}
For a particular material, the transition between phases of different symmetry can be described in terms of an order parameter. Such a parameter represents the extent to which the configuration of the more symmetric phase differs from that of the less symmetric phase.

For the sake of clarity, we fix the dimension to be $d=3$ in the following discussion. To avoid modeling individual liquid crystal molecules, that is very expensive computationally, we pursue a macroscopic description of liquid crystals. Namely, let us describe the orientation of the nematic molecules by a probability distribution in the unit sphere; this gives raise to a tensor field $\vQ : \Omega \to \R^{3\times 3}$, which is required to be symmetric and traceless a.e. \cite{deGennes_book1995,Virga_book1994,Mottram_arXiv2014}.

We can further characterize $\vQ$ by its eigenframe and is often written in the form:
\begin{equation}\label{eqn:Q_matrix_biaxial}
\begin{split}
	\vQ = s_1 (\vn_1 \otimes \vn_1) + s_2 (\vn_2 \otimes \vn_2) - \frac{1}{3} (s_1 + s_2) \vI,
\end{split}
\end{equation}
where $\vn_1$, $\vn_2$ are orthonormal eigenvectors of $\vQ$, with eigenvalues given by
\begin{equation}\label{eqn:Q_matrix_eigenvalues}
\begin{split}
	\lambda_1 = \frac{2 s_1 - s_2}{3}, \quad \lambda_2 = \frac{2 s_2 - s_1}{3}, \quad \lambda_3 = -\frac{s_1 + s_2}{3},
\end{split}
\end{equation}
where $\lambda_3$ corresponds to the eigenvector $\vn_3 \perp \vn_1, \vn_2$. The eigenvalues of $\vQ$ are constrained by
\begin{equation} \label{eqn:restriction_eigenvalues}
-\frac13 \le \lambda_i  \le \frac23, \quad i = 1,2,3.
\end{equation}
 When all eigenvalues are equal, since $\vQ$ is traceless, we must have $\lambda_1 = \lambda_2 = \lambda_3 = 0$ and $s_1 = s_2 = 0$, i.e. the distribution of liquid crystal molecules is isotropic. If two eigenvalues are equal, i.e.
\begin{equation}\label{eqn:Q_matrix_eigenvalues_uniaxial}
\begin{split}
	\lambda_1 &= \lambda_2 \quad \Leftrightarrow \quad s_1 = s_2, \\
	\lambda_1 &= \lambda_3 \quad \Leftrightarrow \quad s_1 = 0, \\
	\lambda_2 &= \lambda_3 \quad \Leftrightarrow \quad s_2 = 0,
\end{split}
\end{equation}
then we encounter a {\it uniaxial} state, in which either molecules prefer to orient in alignment with the simple eigenspace (in case it corresponds to a positive eigenvalue) or perpendicular to it (in case it corresponds to a negative eigenvalue).  If all three eigenvalues are distinct, then the state is called {\it biaxial}. 

\begin{remark}[biaxial nematics]\label{rem:biaxial_vs_uniaxial}
In this work, we regard liquid crystal molecules as elongated rods. Naturally, most liquid crystal molecules do not possess such an axial symmetry. If the molecules resemble a lath more than a rod, it is expected that the energy interaction can be minimized if the molecules are fully aligned; this necessarily involves a certain degree of biaxiality. Roughly, this was the rationale behind the prediction of the biaxial nematic phase by Freiser \cite{Freiser_PRL1970}.

Since that seminal work, empirical evidence of biaxial states in certain lyotropic liquid crystals has been well documented (see \cite{Yu_PRL1980}, for example). Nevertheless, for thermotropic liquid crystals the nematic biaxial phase remained elusive for a long period, and was first reported long after Freiser's original prediction \cite{Madsen_PRL2004,Prasad_JACS2005}. As pointed out by Sonnet and Virga \cite[Section 4.1]{Sonnet_book2012}, 
\begin{quote}
The vast majority of nematic liquid crystals do not, at least in homogeneous equilibrium states, show any sign of biaxiality.
\end{quote}
\end{remark}

We refer to \cite{BorthagarayWalker_sub2019} for further quantitative discussion via computations.
In light of Remark \ref{rem:biaxial_vs_uniaxial}, in Section \ref{sec:formulation} we shall consider a uniaxially-constrained model. More precisely, we assume that $\vQ$ takes the uniaxial state 
\begin{equation}\label{eqn:Q_matrix_uniaxial}
	\vQ = s \left( \vn \otimes \vn - \frac{1}{3} \vI \right),
\end{equation}
where $\vn$ is the main eigenvector with eigenvalue $\lambda = 2s/3$; the other two eigenvalues equal $-s/3$. The scalar field $s$ is called the {\em degree of orientation} of the liquid crystal molecules. Taking into account identities \eqref{eqn:Q_matrix_eigenvalues} and the restrictions \eqref{eqn:restriction_eigenvalues}, it follows that the physically meaningful range is $s \in (-1/2, 1)$. In case $s = 1$, the molecular long axes are in perfect alignment with the direction of $\vn$, whereas $s = -1/2$ represents the state in which all molecules are perpendicular to $\vn$. 

\begin{remark}[problems in $2d$] \label{rmk:QTensor_2d}
The discussion above simplifies considerably when $d=2$. Indeed, since $\vQ$ is symmetric and traceless, it must be uniaxial, and writing it as $\vQ = s \left( \vn \otimes \vn - \frac{1}{2} \vI \right)$, we deduce that its eigenvalues are $\lambda_1 = s/2$, with eigenvector $\vn$, and $\lambda_2 = -\lambda_1$, with eigenvector $\vn^\bot$. Because eigenvalues are constrained to satisfy $\lambda_i \in (-1/2,1/2)$, we deduce that the physically meaningful range is $s \in (-1,1)$. Actually, one can further simplify to $s \in [0,1)$ by noting that a state with director $\vn$ and degree of orientation $s<0$ is equivalent to a state with director $\vm \perp \vn$ and degree of orientation $-s$.
\end{remark}

\begin{remark}[thin films]
For simplicity, in this work we consider $\vQ$ to be a square tensor with the same dimension as the spatial domain. With minor modifications, our approach carries to the case where these dimensions are different, such as three dimensional tensors on thin films.
\end{remark}

\subsection{Continuum Mechanics}\label{sec:continuum_mech}
Given the order parameter $\vQ$, we still need a model to determine its state as a function of space.  For modeling equilibrium states, this amounts to finding minimizers of an energy functional.  A common approach from continuum mechanics \cite{Holzapfel_book2000,Temam_ContMechBook2005,Truesdell_Book1976} is to construct the ``simplest'' functional possible that is quadratic in gradients of the order parameter while obeying standard laws of physics, such as frame indifference and material symmetries. We assume all equations have been non-dimensionalized; see \cite{Gartland_MMA2018} for the case of the Landau - de Gennes model.

\subsubsection{Landau - de Gennes Model}\label{sec:landau-degennes_model}
Using $\vQ$ as the order parameter, we obtain the Landau - de Gennes model, in which the energy is given by \cite{deGennes_book1995,Sonnet_book2012}:
\begin{equation}\label{eqn:Landau-deGennes_energy_example}
\begin{split}
	\ELdG [\vQ] &:= \iO \ELdGfunc (\vQ,\nabla \vQ) \, dx + \frac{1}{\Bulkcoef} \iO \Bulkfunc (\vQ) \, dx, \\
	\ELdGfunc(\vQ,\nabla \vQ) &:= \frac{1}{2} \left( \Li_{1} |\nabla \vQ|^2 + \Li_{2} |\nabla \cdot \vQ|^2 + \Li_{3} (\nabla \vQ)\tp \dd \nabla \vQ \right).
\end{split}
\end{equation}
Above, $\{ \Li_{i} \}_{i=1}^{3}$, $\Bulkcoef$ are material parameters, $\Bulkfunc$ is a bulk (thermotropic) potential and
\begin{equation}\label{eqn:Landau-deGennes_invariants}
\begin{split}
	|\nabla \vQ|^2 := (\partial_{k} Q_{ij}) (\partial_{k} Q_{ij}), \quad |\nabla \cdot \vQ|^2 := (\partial_{j} Q_{ij})^2, \quad
	(\nabla \vQ)\tp \dd \nabla \vQ := (\partial_{j} Q_{ik}) (\partial_{k} Q_{ij}),
\end{split}
\end{equation}
and we use the convention of summation over repeated indices.
This is a relatively simple form for $\ELdGfunc$; more complicated models can also be considered \cite{Mottram_arXiv2014,deGennes_book1995,Sonnet_book2012}.

The bulk potential $\Bulkfunc$ is a double-well type of function that controls the eigenvalues of $\vQ$. The simplest form is given by
\begin{equation}\label{eqn:Landau-deGennes_bulk_potential}
\begin{split}
	\Bulkfunc (\vQ) = \BulkK + \frac{\BulkA}{2} \tr (\vQ^2) - \frac{\BulkB}{3} \tr (\vQ^3) + \frac{\BulkC}{4} \left( \tr (\vQ^2) \right)^2,
\end{split}
\end{equation}
where $\BulkA$, $\BulkB$, $\BulkC$ are material parameters such that $\BulkA$ has no sign, and $\BulkB$, $\BulkC$ are positive; $\BulkK$ is a convenient constant.  It is typical to let $\BulkA \leq 0$ since we are interested in uniaxial states, so throughout this paper we assume that
\begin{equation}\label{eqn:Landau-deGennes_bulk_params}
\begin{split}
	\BulkA \leq 0, \quad \BulkB, \BulkC > 0,
\end{split}
\end{equation}
which implies that $\Bulkfunc (\vQ) \geq 0$ assuming $\BulkK$ is suitably chosen.

In two dimensions, $\tr(\vQ^3) = 0$, because $\vQ^2 = \frac{s^2}{4} \vI$.  Hence, $\BulkB$ is irrelevant in $2d$, and it is necessary that $\BulkA$ be strictly negative in order to have a stable nematic phase.  This also implies that $\Bulkfunc$ is an \emph{even} function of $s$ if $\vQ$ is uniaxial (see Remark \ref{rmk:QTensor_2d}).

As a simplification, one can take $\Li_{1} = 1$, $\Li_{2} = \Li_{3} = 0$ in \eqref{eqn:Landau-deGennes_energy_example} to obtain a \emph{one-constant approximation}
\begin{equation}\label{eqn:Landau-deGennes_energy_one_const}
\begin{split}
\ELdGone [\vQ] &:= \frac{1}{2} \iO |\nabla \vQ|^2 \, dx + \frac{1}{\Bulkcoef} \iO \Bulkfunc (\vQ) \, dx,
\end{split}
\end{equation}

\subsubsection{Ericksen Model}\label{sec:ericksen_model}

Though the Landau - de Gennes model is quite general, it can be fairly expensive when $d = 3$.  In such a case, since $\vQ \in \R^{3 \times 3}$ and symmetric, it has five independent variables.  Moreover, the bulk potential $\Bulkfunc$ is a non-linear function of $\vQ$, which couples all five variables together when seeking a minimizer of $\ELdG$.

Assuming that $\vQ$ is uniaxial \eqref{eqn:Q_matrix_uniaxial}, we can take $s$ \emph{and} $\vn$ as order parameters.  
In the same way as \eqref{eqn:Landau-deGennes_energy_one_const}, we have a \emph{one-constant} Ericksen model:
\begin{equation}\label{eqn:Ericksen_energy_TEMP_one_const}
\begin{split}
\Eerkone [s,\vn] &:= \frac{\kappa}{2} \iO |\nabla s|^2 \, dx + \frac{1}{2} \iO s^2 |\nabla \vn|^2 \, dx + \frac{1}{\Bulkcoef} \iO \DWfunc(s) \, dx,
\end{split}
\end{equation}
where $\kappa > 0$ is a single material parameter, and $\DWfunc$ is a double-well potential acting on $s$, which is taken from the Landau - de Gennes case: $\DWfunc(s) = \Bulkfunc (\vQ(s))$, where $\vQ$ is any matrix having the form \eqref{eqn:Q_matrix_uniaxial}.

\begin{remark}[Oseen-Frank model]
In case the degree of orientation is a non-zero constant field, the energy $\Eerkone$ effectively reduces to the Oseen-Frank energy \cite{Virga_book1994}: $E_{\mathrm{OF}}[\vn] := \int_\Omega |\nabla \vn|^2$. The Oseen-Frank model has been used extensively in the modeling of liquid crystal-based flat panel displays. Minimizers of the one-constant energy in such a model are director fields $\vn \colon \Om \to \Sp^{d-1}$ satisfying $\Delta \vn - \lambda \vn = 0$, where $\lambda$ is a Lagrange multiplier that enforces the unit length constraint.

In the Oseen-Frank model, point defects in three dimensional domains have finite energy. However, this model is incapable of capturing higher-dimensional defects, that is, defects supported either on lines or planes. Since these naturally occur in many liquid crystal systems, this is a major inherent limitation of the Oseen-Frank model.
\end{remark}

We point out that \eqref{eqn:Ericksen_energy_TEMP_one_const} is \emph{degenerate,} in the sense that $s$ may vanish; this allows for $\vn$ to have discontinuities (i.e. defects) with finite energy.  Indeed, the hallmark of this model is to regularize defects using $s$, but still retain part of the Oseen-Frank model. Discontinuities in $\vn$ may still occur in the singular set
\begin{equation}
    \label{eqn:singular_set}
\Sing := \{ x \in \Om :\; s(x) = 0 \}.
\end{equation}

For problems in $\R^3$, because $\vn \in \Sp^{2}$, it is uniquely defined by two parameters. Thus, in such a case the Ericksen model only has three scalar order parameters, as opposed to five in the Landau - de Gennes model. Another advantage of the Ericksen model is that $s$ and $\vn$ provide a natural way to split the system which is convenient for numerical purposes.
Additionally, the parameter $\kappa$ in \eqref{eqn:Ericksen_energy_TEMP_one_const} plays a major role in the occurrence of defects. Assuming that $s$ equals a sufficiently large positive constant on $\dOm$, if $\kappa$ is large, then $\iO \kappa |\nabla s|^2 dx$ dominates the energy and $s$ stays close to such a positive constant within the domain $\Om$. Thus, defects are less likely to occur. If $\kappa$ is small (say $\kappa < 1$), then $\iO s^2 |\nabla \vn|^2 dx$ dominates the energy, and $s$ may vanish in regions of $\Om$ and induce a defect.  This is confirmed by the numerical experiments in \cite{NochettoWalker_MRSproc2015,Nochetto_SJNA2017}.

\begin{remark}[orientability]
\label{rem:vector_vs_line_field}
Director field models --either Oseen-Frank or Ericksen-- are more than adequate in some situations, although in general they introduce a {\em nonphysical orientational bias} into the problem. Even though liquid crystal molecules may be polar, in nematics one always finds that the states with $\vn$ and $-\vn$ are equivalent \cite{Gramsbergen_PR1986}. At the molecular level, this means that the same number of molecules point ``up'' and ``down.'' Therefore, \emph{line-fields} are more appropriate for modeling nematic liquid crystals.

Another issue with the use of the vector field $\vn$ as an order parameter instead of the matrix $\vQ$ is that the only allowable defects in such a case are \emph{integer order} defects. On the other hand $\vQ$, specifically $\vn \otimes \vn$ in \eqref{eqn:Q_matrix_uniaxial}, is able to represent \emph{line fields} having half-integer defects.  These have been largely observed and documented in experiments; see for example \cite{Brinkman_PT1982,Ohzono_SR2017} and references therein. We point out that, if a line field is \emph{orientable}, then a vector field representation is essentially equivalent \cite{Ball_PAMM2007,Ball_ARMA2011}.
\end{remark}

\section{Mathematical formulation}\label{sec:formulation}

In this work, we will be concerned with the one-constant energy for $\vQ$, given by \eqref{eqn:Landau-deGennes_energy_one_const}. Enforcing $\vQ$ to be symmetric and traceless, one can, in principle, directly minimize such an energy. For three-dimensional problems, a standard approach to finding minimizers \cite{Araki_PRL2006, Tojo_EPJE2009, Sonnet_book2012, Kim_JPCM2013} is to express $\vQ(x)$ as
\begin{equation}\label{eqn:Q-tensor_order_parameter}
\vQ(x) =
\left[
\begin{array}{ccc}
q_1 & q_3 & q_4 \\
q_3 & q_2 & q_5 \\
q_4 & q_5 & -(q_1 + q_2) \\
\end{array}
\right],
\end{equation}
i.e. minimize \eqref{eqn:Landau-deGennes_energy_one_const} with respect to the order parameters $\{ q_i(x) \}^5_{i=1}$. This approach has two drawbacks.

First, a basic argument shows that minimizers of $\iO \Bulkfunc$ have the form of a uniaxial nematic \eqref{eqn:Q_matrix_uniaxial} \cite{Sonnet_book2012}. This is \emph{false} for $\ELdGone$ in \eqref{eqn:Landau-deGennes_energy_one_const} with general boundary conditions. Thus, minimizers of the form \eqref{eqn:Q-tensor_order_parameter} \emph{violate} the algebraic form of \eqref{eqn:Q_matrix_uniaxial} and exhibit a {\em biaxial escape} \cite{Palffy-muhoray_LC1994, Sonnet_PRE1995, Lamy_arXiv2013}. This is analogous to the escape to the 3rd dimension in liquid crystal director models \cite{Virga_book1994}. This is not desirable if the underlying nematic liquid crystal is guaranteed to be uniaxial (recall Remark \ref{rem:biaxial_vs_uniaxial}).
Secondly, minimizing \eqref{eqn:Landau-deGennes_energy_one_const} with $\vQ$ of the form \eqref{eqn:Q-tensor_order_parameter} leads to a non-linear system with five coupled variables in 3d, so it is expensive to solve and possibly not robust \cite{Lee_APL2002,Ravnik_LC2009,Zhao_JSC2016,Zhao_JCP2016}.

These drawbacks motivate us to enforce the uniaxiallity constraint \eqref{eqn:Q_matrix_uniaxial} directly in the Landau - de Gennes one-constant energy \eqref{eqn:Landau-deGennes_energy_one_const}. The ensuing model has similarities with the Ericksen model \eqref{eqn:Ericksen_energy_TEMP_one_const}, although it has the advantage of allowing minimizers to exhibit half-integer order defects. Our approach hinges on previous work on the Ericksen model \cite{NochettoWalker_MRSproc2015,Nochetto_JCP2018,Nochetto_SJNA2017}, which exploits a hidden structure of \eqref{eqn:Ericksen_energy_TEMP_one_const}. We next reveal such structure for the Landau - de Gennes model with uniaxial constraint and point out the corresponding counterpart for the Ericksen model when appropriate. Compared to directly minimizing \eqref{eqn:Landau-deGennes_energy_one_const} using  \eqref{eqn:Q-tensor_order_parameter}, our algorithm finds a minimizer by \emph{solving a sequence of linear systems of smaller dimension}. However, our approach is equivalent to directly minimizing the energy \eqref{eqn:Landau-deGennes_energy_one_const} for two-dimensional problems (see Remark \ref{rmk:QTensor_2d}).

\subsection{The Basic Structure} \label{sec:basic-structure}
We start with the main part (elastic energy) of the one-constant Ericksen model in \eqref{eqn:Ericksen_energy_TEMP_one_const}, namely
\begin{equation} \label{eqn:Ericksen_elastic_energy_one_const}
    \Eerkmain[s,\vn] := \frac{1}{2} \iO \Big(\kappa|\nabla s|^2 \, dx + s^2 |\nabla \vn|^2 \Big) \, dx.
\end{equation}
It is clear that a configuration $(s,\vn)$ with finite elastic energy implies $s\in H^1(\Om)$ and that the weight $s$ vanishing within the singular set $\Sing$ of \eqref{eqn:singular_set} allows for director fields $\vn$ with infinite Dirichlet energy and thus for the presence of defects. The hidden structure in \eqref{eqn:Ericksen_elastic_energy_one_const} becomes apparent upon introducing the auxiliary variable $\vu = s \vn$ as proposed first in \cite{Ambrosio_MM1990a, Lin_CPAM1991}: since $| \vn | = 1$ we get $\nabla\vn \, \vn = \vzero$ and the pointwise orthogonal decomposition $\nabla \vu = \vn \otimes \nabla s + s \nabla \vn$. Consequently, \eqref{eqn:Ericksen_elastic_energy_one_const} can be equivalently written
\begin{equation}\label{eqn:Ericksen_auxiliary_energy_identity}
  \Eerkmain[ s , \vn] = \EerkUmain[s,\vu] := \frac12 \iO \Big((\kappa - 1) | \nabla s |^2 + | \nabla  \vu |^2 \Big) dx,
\end{equation}
to discover that $\vu \in [H^1(\Omega)]^d$. Moreover, it is apparent from \eqref{eqn:Ericksen_auxiliary_energy_identity} that if $\kappa > 1$ the Ericksen energy $\EerkUmain[s,\vu]$ is convex with respect to $(s,\vu)$. The physically relevant case $0 < \kappa < 1$ in terms of defects is more difficult with regard to proving $\Gamma$-convergence, because convexity of $\EerkUmain[s,\vu]$ is no longer obvious unless we exploit the relation $|s|=|\vu|$. This relation can only be enforced at the nodes of a finite element approximation of $(s,\vu)$, whence convexity as well as weak lower semi-continuity of $\EerkUmain[s,\vu]$  become problematic \cite{NochettoWalker_MRSproc2015,Nochetto_SJNA2017,Nochetto_JCP2018}; we will refer to this issue later in Lemma \ref{lem:LdG_wlsc}.

We now turn to the Landau - de Gennes model with uniaxial constraint \eqref{eqn:Q_matrix_uniaxial}. To this end, we introduce the line field $\vNN = \vn \otimes \vn$, which will be treated as a control variable in minimizing \eqref{eqn:Landau-deGennes_energy_one_const} subject to \eqref{eqn:Q_matrix_uniaxial}. Since $\nabla \vQ$ is a 3-tensor of the form $\nabla \vQ = \nabla s \otimes \left( \vNN - \frac1d \vI \right) + s \nabla \vNN$, we have
\[
 |\nabla \vQ|^2 = |\nabla s|^2 \left| \vNN - \frac1d \vI \right|^2 + s^2 |\nabla \vNN|^2 + 2 s \left[ \nabla s \otimes \left( \vNN - \frac1d \vI \right) \right] \dd \nabla \vNN.
 \]
 A direct calculation gives $\left| \vNN - \frac1d \vI \right|^2 = \frac{d-1}{d}$ and $\left[ \nabla s \otimes \left( \vNN - \frac1d \vI \right) \right] \dd \nabla \vNN = 0$ because $\nabla\vNN : \vNN = \nabla\vNN : \vI = \vzero$. Therefore, we obtain the first relation with the Ericksen model
\begin{align*}
    |\nabla \vQ|^2 = \frac{d-1}{d} |\nabla s|^2  + s^2 |\nabla \vNN|^2.
\end{align*}
The second one comes from the equalities
\begin{equation*}
   s^2 = C_2 \, \tr(\vQ^2), \quad s^3 = C_3 \, \tr(\vQ^3), \quad s^4 = C_4 \, (\tr(\vQ^2))^2,
\end{equation*}
which are valid for suitable constants $C_2,C_3,C_4 > 0$. Consequently, the double-well potential $\Bulkfunc (\vQ)$ in \eqref{eqn:Landau-deGennes_bulk_potential} becomes a quartic function $\BulkLdG(s) = \Bulkfunc (\vQ)$ of $s$ that blows-up at the end points of the interval $[-\frac{1}{d-1}, 1]$ and forces $s$ to remain within this physical range. If we let the main energy be
\begin{equation}\label{eqn:main-energy}
  \Eunimain[s,\vNN]:=\Eunis[s]+\Euniring[s,\vNN],
\end{equation}
where the orientation, interaction and bulk energies are given by
\[
\Eunis[s] := \frac{d-1}{2d} \iO |\nabla s|^2,
\quad
\Euniring[s,\vNN] := \frac12 \iO s^2 |\nabla \vNN|^2 \, dx,
\quad
\Ebulk[s] := \frac{1}{\Bulkcoef} \iO \BulkLdG (s) \, dx,
\]
then the Landau - de Gennes total energy $\Eunit[s,\vNN]=\ELdGone [\vQ]$ in \eqref{eqn:Landau-deGennes_energy_one_const} reads
\begin{equation}\label{eqn:nematic_Q-tensor_energy_s_ntens}
\Eunit[s,\vNN] = \Eunimain[s,\vNN] + \Ebulk[s].
\end{equation}
We see that this energy has the \emph{same form} as the Ericksen energy \eqref{eqn:Ericksen_energy_TEMP_one_const}, except that $\vNN$ replaces $\vn$ and $\kappa = (d-1)/d < 1$. This motivates a change of variable analogous to the one in the Ericksen model: we set $\vU := s \vNN$ and note that $\nabla\vU = \nabla s \otimes\vNN + s \nabla\vNN$ is a $d$-tensor with orthogonal components, whence
\[
|\nabla\vU|^2 = |\nabla s|^2 + s^2 |\nabla\vNN|^2
\]
and the main and total energies in terms of $(s,\vU)$ read
\begin{gather}\label{eqn:Q_tensor_energy_sU}
  \Eunimain[s,\vNN] = \EuniUmain[s,\vU] := -\frac{1}{2d} \iO  |\nabla s|^2 \, dx + \frac12 \iO |\nabla \vU|^2 \, dx,
  \\
  \label{eqn:total-energy-U}
  \EuniUtotal[s,\vU] := \EuniUmain[s,\vU] + \Ebulk[s].
\end{gather}
Similarly, we could set $\widetilde{s} := |s|$ and $\widetilde\vU := \widetilde{s}\vNN$ to arrive at $\EuniUmain[\widetilde{s},\widetilde\vU] = \EuniUmain[s,\vU]$ because $|\nabla\widetilde{s}| = |\nabla s|$ a.e. in $\Om$.
We are now able to reach similar conclusions as for the Ericksen model. If $\ELdGone [\vQ]<\infty$, then $(s,\vU)\in H^1(\Omega) \times [H^1(\Omega)]^{d\times d}$ but in general $\vNN \notin [H^1(\Omega)]^{d\times d}$ because the presence of the weight $s^2$ in $\Euniring[s,\vNN]$ allows for blow-up of $\nabla\vNN$ in the singular set $\Sing$ of \eqref{eqn:singular_set}. We intend to preserve this basic structure discretely. In fact, this will be crucial later in Section \ref{sec:LdG_Gamma_conv} to interpret $\nabla\vNN$ in the Lebesgue $L^2$ sense and recover the orthogonality relation $|\nabla\vU|^2 = |\nabla s|^2 + s^2 |\nabla\vNN|^2$ a.e. in $\Omega\setminus\Sing$, as well as to derive $\Gamma$ convergence.

In order to define the {\it admissible class} of functions, we begin with the set of {\it line fields}
\begin{equation} \label{eqn:line-fields}
\LL^{d-1} := \{ \vA \in \R^{d\times d} \colon \mbox{ there exists } \vn \in \Sp^{d-1}, \ \vA = \vn \otimes \vn \}.
\end{equation}
We say that a triple $(s,\vNN,\vU)$ satisfies the {\it structural condition} provided 
\begin{equation}\label{eqn:LdG_structural_conditions}
  -\frac{1}{d-1} \leq s \leq 1, \qquad \vU = s \vNN, \qquad \vNN \in \LL^{d-1}
  \qquad \text{a.e. } \Om.
\end{equation}
We next define the admissible class of functions to be
\begin{equation}\label{eqn:LdG_admissibleclass}
\Admisuni := \big \{ (s, \vNN, \vU) \in H^1(\Om) \times [L^\infty(\Om)]^{d \times d} \times [H^1(\Om)]^{d \times d}: (s,\vNN,\vU)  \text{ satisfies \eqref{eqn:LdG_structural_conditions}} \big \},
\end{equation}
To enforce boundary conditions, let $(\Gamma_s,\Gamma_\vNN,\Gamma_\vU)$ with $\Gamma_\vNN=\Gamma_\vU$ be open subsets of $\dOm$ where we impose Dirichlet conditions. Given functions $(g,\vM,\vR) \in W^1_\infty(\R^d) \times [L^\infty(\R^d)]^{d\times d} \times [W^1_\infty(\R^d)]^{d \times d}$ that satisfy the structural condition \eqref{eqn:LdG_structural_conditions} in a neighborhood of $\dOm$, we define the restricted admissible class
\begin{align}\label{eqn:LdG_restricted_admissible_class}
  \Admisuni(g,\vR) := \big\{ (s, \vNN, \vU) \in \Admisuni : \ s|_{\Gamma_s} = g, \quad \vU|_{\Gamma_\vU} = \vR \big\},
\end{align}
Moreover, we assume that for some $\delta_0 > 0$
\begin{equation} \label{eqn:g_nondegenerate}
-\frac{1}{d-1} + \delta_0 \le g \le 1-\delta_0 \quad \mbox{ in } \Om,
\end{equation}
and
\begin{equation} \label{eqn:g_geq_0}
g \ge \delta_0 \ \mbox{ on } \dOm, 
\end{equation}
so that the function $\vM$ is of class $W^1_\infty$ in a neighborhood of $\Gamma_\vNN$ and satisfies $\vM = g^{-1} \vR \in \LL^{d-1}$ on $\Gamma_\vNN$.

Finally, we assume that the coefficients $\BulkA, \BulkB, \BulkC$ in 
\eqref{eqn:Landau-deGennes_bulk_potential} are such that
\begin{equation} \label{eqn:psi_nondegenerate} 
\BulkLdG(s) \ge \BulkLdG(1-\delta_0) \ \mbox{ for } s \ge 1- \delta_0, \quad \BulkLdG(s) \ge \BulkLdG\left(-\frac{1}{d-1}+\delta_0\right) \ \mbox{ for } s \le -\frac{1}{d-1}+ \delta_0.
\end{equation}
This will lead to confinement of $s$ with the interval $[-\frac{1}{d-1}+ \delta_0,1- \delta_0]$.

\section{Discretization} \label{sec:LdG_discretization}

Let $\Tk_h = \{ T \}$ be a conforming shape-regular and quasi-uniform triangulation of $\Om$ made of simplices. Let $\Nk_h=\{x_i\}_{i=1}^N$ be the set of nodes (vertices) $x_i$ of $\Tk_h$ and $N$ be its cardinality. Let $\phi_i$ be the standard ``hat'' basis function associated with the node $x_i \in \Nk_h$. We indicate with $\omega_i = \text{supp} \;\phi_i$ the patch of a node $x_i$ (i.e. the ``star'' of elements in $\Tk_h$ that contain the vertex $x_i$). For simplicity we assume that $\Om=\Om_h$, so that there is no geometric error caused by domain approximation. We further assume that $\Tk_h$ is {\it weakly acute}, namely
\begin{equation}\label{eqn:weakly-acute}
  k_{ij} := -\int_{\Om} \nabla \phi_i \cdot \nabla \phi_j \, dx \geq 0 \quad\text{for all } i\ne j.
\end{equation} 
Condition \eqref{eqn:weakly-acute} ensures the validity of the discrete maximum principle.
However, \eqref{eqn:weakly-acute} imposes a severe geometric restriction on $\Tk_h$ \cite{Ciarlet_CMAME1973, Strang_FEMbook2008}, especially in three dimensions. 

We consider three continuous piecewise linear Lagrange finite element spaces on $\Om$:
\begin{equation}\label{eqn:LdG_discrete_spaces}
\begin{split}
  \Sh &:= \{ s_h \in H^1(\Om) : s_h |_{T} \text{ is affine for all } T \in \Tk_h \}, \\
\Uh &:= \{ \vU_h \in [H^1(\Om)]^{d\times d} : \text{ each entry of } \vU_h |_{T} \text{ is afffine for all } T \in \Tk_h \}, \\
\THh &:= \{ \vNN_h \in \Uh : \vNN_h(x_i) \in \LL^{d-1}, \text{ for all } x_{i} \in \Nk_h  \},
\end{split}
\end{equation}
where $\THh$ imposes both the rank-one and unit-norm constraints only at the vertices of the mesh $\Tk_h$. We say that the discrete triple $(s_h,\vNN_h,\vU_h) \in \Sh \times \THh \times \Uh$ satisfies the {\it discrete structural condition} if
\begin{equation}\label{eqn:LdG_struct_condition_discrete}
\vU_h = I_h (s_h \vNN_h), \qquad - \frac{1}{d-1} \leq s_h \leq 1,
\end{equation}
where $I_h$ stands for the Lagrange interpolation operator. All such triples make the {\it discrete admissible set} $\Admisuni^{h}$. We let $g_h := I_h g$, $\vR_h := I_h \vR$, and $\vM_h := I_h \vM$ be the discrete Dirichlet data, and incorporate Dirichlet boundary conditions within the discrete spaces:
\begin{equation*}\label{eqn:LdG_discrete_spaces_BC}
\begin{split}
	\Sh (g_h) &:= \{ s_h \in \Sh : s_h |_{\Gamma_s} = g_h \}, \\
	\Uh (\vR_h) &:= \{ \vU_h \in \Uh : \vU_h |_{\Gamma_\vU} = \vR_h \}, \\
	\THh (\vM_h) &:= \{ \vNN_h \in \THh : \vNN_h |_{\Gamma_\vNN} = \vM_h \}.
\end{split}
\end{equation*}
In view of \eqref{eqn:g_geq_0}, the following compatibility condition must hold:
$\vM_h = I_h[g_h^{-1} \vR_h]$ on $\Gamma_\vNN$.
This leads to the following discrete admissible class with boundary conditions:
\begin{equation}\label{eqn:LdG_admissible_class_BC_discrete}
	\Admisuni^{h}(g_h,\vR_h) := \big\{(s_h,\vNN_h,\vU_h) \in \Sh(g_h)\times\THh(\vM_h)\times\Uh(\vR_h)): \quad (s_h,\vU_h,\vNN_n) \text{ satisfies \eqref{eqn:LdG_struct_condition_discrete}} \big\},
\end{equation}

We are now ready to introduce the discrete version of $\Eunimain[s, \vNN]$ which mimics that of the Ericksen model \cite{NochettoWalker_MRSproc2015,Nochetto_SJNA2017,Nochetto_JCP2018}. First note that
$
\sum_{j=1}^N k_{ij} = 0
$
for all $x_i \in \Nk_h$, and for $s_h = \sum_{i=1}^N s_h(x_i) \phi_i \in\Sh$ we have
\begin{align*}
  \int_{\Om} | \nabla s_h |^2 dx = -\sum_{i=1}^N k_{ii} s_h(x_i)^2 - \sum_{i, j = 1, i \neq j}^N k_{ij} s_h(x_i) s_h(x_j).
\end{align*}
Using $k_{ii} = - \sum_{j \neq i} k_{ij}$ and the symmetry $k_{ij}=k_{ji}$, we thus obtain
\begin{equation}\label{eqn:dirichlet_integral_identity}
\begin{aligned}
  \int_{\Om} | \nabla s_h |^2 dx 
= \frac{1}{2} \sum_{i, j = 1}^N k_{ij} \big( \dij s_h \big)^2,
\end{aligned}
\end{equation}
where we have introduced the notation
\begin{equation}\label{eqn:delta_ij}
\dij s_h := s_h(x_i) - s_h(x_j), \quad \dij \vNN_h := \vNN_h(x_i) - \vNN_h(x_j).
\end{equation}
We next define the main part of the discrete Landau - de Gennes energy to be
\begin{equation}\label{eqn:LdG_discrete_energy_main}
\begin{split}
\Eunimain^h[s_h, \vNN_h] := & \frac{d-1}{4d} \sum_{i, j = 1}^N k_{ij} \left( \dij s_h \right)^2 + \frac{1}{4} \sum_{i, j = 1}^N k_{ij} \left(\frac{s_h(x_i)^2 + s_h(x_j)^2}{2}\right) |\dij \vNN_h|^2.
\end{split}
\end{equation}
We point out that the first term corresponds to
\[
\Eunis^h [s_h] = \frac{d-1}{2d} \iO |\nabla s_h|^2 dx = 
\frac{d-1}{4d} \sum_{i, j = 1}^N k_{ij} \left( \dij s_h \right)^2,
\]
while the second term is a first order nonstandard approximation of $\Euniring[s,\vNN]=\frac12\iO s^2|\nabla \vNN|^2 dx$,
\begin{equation} \label{eqn:LdG_def_Euniringh}
\Euniring^h[s_h,\vNN_h] := \frac14 \sum_{i, j = 1}^N k_{ij} \left(\frac{s_h(x_i)^2 + s_h(x_j)^2}{2}\right) |\dij \vNN_h|^2
\end{equation}
introduced in \cite{Nochetto_SJNA2017}. As we will see below, a key feature of this discretization is that it makes it possible to handle degenerate parameters $s_h$ without regularization. This is due to Lemma \ref{lem:energy_inequality}, which deals with discrete versions of $\EuniUmain[s,\vU]$ defined in \eqref{eqn:Q_tensor_energy_sU} involving the auxiliary variable $\vU_h$:
\begin{equation}\label{eqn:EuniUmain}
\EuniUmain^h[s_h,\vU_h] := -\frac{1}{2d} \iO |\nabla s_h|^2 dx + \frac12\iO |\nabla \vU_h|^2 dx.
\end{equation}

We finally discretize the nonlinear bulk energy in the usual manner
\begin{equation}\label{eqn:LdG_discrete_energy_bulk}
\Ebulk^h [s_h] := \frac{1}{\Bulkcoef} \iO \BulkLdG (s_h) dx.
\end{equation}
With the notation introduced above, the formulation of the discrete problem is as follows: find $(s_h, \vNN_h) \in \Sh (g_h) \times \THh (\vM_h)$ such that the following discrete total energy is minimized:
\begin{equation}\label{eqn:LdG_discrete_energy}
\Eunit^h[s_h, \vNN_h] := \Eunimain^h [s_h, \vNN_h] + \Ebulk^h [s_h].
\end{equation}

Because the discrete spaces consist of piecewise linear functions, the structural condition $\vU_h = s_h \vNN_h$ is only satisfied at the mesh nodes (cf. \eqref{eqn:LdG_struct_condition_discrete}). Therefore, there is a variational crime that we need to account for. To this end, we now derive energy inequalities similar to \cite[Lemma 2.2]{Nochetto_SJNA2017}. Although the arguments are the same, we present the proof for completeness. For our analysis, we introduce the functions
\begin{equation} \label{eqn:def_of_tilde}
\widetilde{s}_h = I_h(|s_h|), \qquad \widetilde{\vU}_h = I_h (|s_h| \vNN_h),
\end{equation}
and remark that $(\widetilde{s}_h, \vNN_h, \widetilde{\vU}_h)$ satisfies \eqref{eqn:LdG_struct_condition_discrete}.

\begin{lemma}[energy inequality] \label{lem:energy_inequality}
  Let the mesh $\Tk_h$ satisfy \eqref{eqn:weakly-acute}. Then, for all $(s_h,\vNN_h,\vU_h) \in \Admisuni^h(g_h,\vR_h)$, the main part of the discrete Landau - de Gennes energy satisfies
\begin{equation}\label{eqn:energyequality}
  \Eunimain^h[s_h, \vNN_h] - \EuniUmain^h[s_h,\vU_h] = \consistLdG_h,
\end{equation}
as well as
\begin{equation} \label{eqn:energyequality_tilde}
  \Eunimain^h[s_h, \vNN_h] - \EuniUmain^h[\widetilde{s}_h,\widetilde{\vU}_h]
  \ge \consistULdG_h,
\end{equation}
where $\EuniUmain^h[s_h,\vU_h]$ is defined in \eqref{eqn:EuniUmain} and
\begin{equation}\label{eqn:LdG_residual}
  \consistLdG_h := \frac{1}{8} \sum_{i, j = 1}^N k_{ij} \big(\dij s_h
  \big)^2 \big| \dij \vNN_h \big|^2 \ge 0,
  \qquad
  \consistULdG_h := \frac{1}{8} \sum_{i, j = 1}^N k_{ij}
  \big(\dij \widetilde{s}_h \big)^2 \big| \dij \vNN_h \big|^2 \ge 0.
\end{equation}
\end{lemma}
\begin{proof}
Expanding 
\[ \begin{aligned}
s_h(x_i) \vNN_h(x_i) - s_h(x_j) \vNN_h(x_j) = &  \frac{s_h(x_i) + s_h(x_j)}{2} \, \dij \vNN_h + \frac{\vNN_h(x_i) + \vNN_h(x_j)}{2} \, \dij s_h 
\end{aligned} \]
and using the orthogonality relation $(\dij \vNN_h): \big(\vNN_h(x_i) + \vNN_h(x_j) \big) = 0$,
we can write
\[
\frac12 \iO |\nabla \vU_h|^2 = \frac14 \sum_{i,j = 1}^N k_{ij} \left( \frac{s_h(x_i) + s_h(x_j)}{2} \right)^2 |\dij \vNN_h|^2 + 
\frac14 \sum_{i,j = 1}^N k_{ij} (\dij s_h)^2 \left| \frac{\vNN_h(x_i) + \vNN_h(x_j)}{2} \right|^2.
\]
Next, we utilize the identities $\big(s_h(x_i) + s_h(x_j)\big)^2 = 2 \big(s_h(x_i)^2 + s_h(x_j)^2\big) - \big(s_h(x_i) - s_h(x_j)\big)^2$ and 
$\big|\vNN_h(x_i) + \vNN_h(x_j)\big|^2 = 4- |\dij \vNN_h|^2$, to obtain
\begin{equation} \label{eqn:LdG_gradvU_h}
\frac12 \iO |\nabla\vU_h|^2 dx = \frac14 \sum_{i,j = 1}^N k_{ij} \left( \frac{s_h(x_i)^2 + s_h(x_j)^2}{2} \right) |\dij \vNN_h|^2 + \frac14 \sum_{i,j = 1}^N k_{ij} (\dij s_h)^2 - \consistLdG_h.
\end{equation}
Identity \eqref{eqn:energyequality} follows immediately. 

On the other hand, repeating the argument above  with $(\widetilde{s}_h, \widetilde{\vU}_h)$ instead of $(s_h,\vU_h)$ gives
\begin{equation}  \label{eqn:LdG_gradvU_h_tilde}
\frac12 \iO |\nabla\widetilde{\vU}_h|^2 dx = \frac14 \sum_{i,j = 1}^N k_{ij} \left( \frac{\widetilde{s}_h(x_i)^2 + \widetilde{s}_h(x_j)^2}{2} \right) |\dij \vNN_h|^2 + \frac14 \sum_{i,j = 1}^N k_{ij} (\dij \widetilde{s}_h)^2 - \consistULdG_h.
\end{equation}
This yields
\[
\EuniUmain^h[\widetilde{s}_h,\widetilde{\vU}_h] = \Eunimain^h[\widetilde{s}_h,\vNN_h] - \consistULdG_h.
\] 
The properties $\consistLdG_h \ge 0$ and $\consistULdG_h \ge 0$ are a consequence of the mesh acuteness assumption \eqref{eqn:weakly-acute}. Moreover, since $|\dij \widetilde{s}_h| \le |\dij s_h|$ and $\widetilde{s}_h (x_i)^2 = s_h (x_i)^2$ for all $i,j = 1, \ldots, N$, we have 
\begin{equation} \label{eqn:bound_nabla_sh}
\| \nabla \widetilde{s}_h \|_{L^2(\Om)} = \frac12 \sum_{i,j = 1}^N k_{ij} (\dij \widetilde{s}_h)^2 \le \frac12 \sum_{i,j = 1}^N k_{ij} (\dij s_h)^2 = \| \nabla s_h \|_{L^2(\Om)}.
\end{equation}
Therefore, $\Eunimain^h[\widetilde{s}_h,\vNN_h] \le \Eunimain^h[s_h,\vNN_h]$ and inequality \eqref{eqn:energyequality_tilde} follows.
\end{proof}

\section{$\Gamma$-convergence of the discrete energies} \label{sec:LdG_Gamma_conv}

This section shows that the discrete problems \eqref{eqn:LdG_discrete_energy} $\Gamma$-converge to the continuous problem \eqref{eqn:nematic_Q-tensor_energy_s_ntens}. We set the continuous and discrete spaces
\[
\X := L^2(\Omega) \times [L^2(\Omega)]^{d\times d}\times [L^2(\Omega)]^{d\times d},
\qquad \X_h := \Sh \times \THh \times \Uh,
\]
and define $\Eunit [ s, \vNN]$ as in \eqref{eqn:nematic_Q-tensor_energy_s_ntens} if $(s,\vNN) \in \Admisuni(g,\vR)$ and $\Eunit [s,\vNN] = \infty$ if $(s,\vNN) \in \X \setminus \Admisuni(g,\vR)$. In a similar fashion, we define $\Eunit^h[s_h,\vNN_h]$ as in \eqref{eqn:LdG_discrete_energy} if $(s_h,\vNN_h) \in \Admisuni^h(g_h,\vR_h)$ and $\Eunit^h [s_h,\vNN_h] = \infty$ if $(s_h,\vNN_h) \in \X_h \setminus \Admisuni(g_h,\vR_h)$.

\subsection{Lim-sup property: Existence of a recovery sequence}
\label{sec:LdG_limsup}

Our goal is to show the following property: given $(s,\vNN,\vU) \in \X$, there exists a sequence $(s_h,\vNN_h,\vU_h) \in \Admisuni^h(g_h,\vM_h)$ such that 
\begin{equation} \label{eqn:LdG_convergence_H1}
  \| (s,\vU) - (s_h, \vU_h) \|_{H^1(\Om)} \to 0,
  \quad
  \| \vNN - \vNN_h \|_{L^2(\Omega\setminus\Sing)} \to 0,
\end{equation}
as $h\to0$ and
\begin{equation} \label{eqn:LdG_limsup}
\limsup_{h \to 0} \Eunit^h[s_h, \vNN_h] \le \Eunit[s,\vNN],
\end{equation}
where $\Eunit[s,\vNN]$ is defined in \eqref{eqn:nematic_Q-tensor_energy_s_ntens}.

\medskip\noindent
{\bf Truncation.}
Naturally, the interesting case to consider is when $(s, \vNN) \in \Admisuni(g,\vM)$; otherwise the property above is trivially true. As shown in \cite[Lemma 3.1]{Nochetto_SJNA2017}, hypotheses \eqref{eqn:g_nondegenerate} and \eqref{eqn:psi_nondegenerate} make it possible to assume that the degree of orientation $s$ is sufficiently far from the boundary of the physically meaningful range $[-\frac{1}{d-1},1]$. We state this precisely next.

\begin{lemma}[truncation]
Given $(s,\vNN,\vU) \in \Admisuni(g,\vR)$, let $(\hat{s},\widehat\vU)$ be the truncations
\[
\hat{s}(x) := \min \left\{ 1- \delta_0, \max \left\{ -\frac{1}{d-1} + \delta_0, \, s(x) \right\} \right\},
\quad
\widehat\vU := \hat s \vNN.
\]
Then, $(\hat{s},\vNN,\widehat{\vU}) \in \Admisuni(g,\vR)$ and
\[
\Eunimain[\hat{s},\vNN] \le \Eunimain[s,\vNN], \qquad \Ebulk[\hat{s},\vNN] \le \Ebulk[s,\vNN].
\]
Moreover, given $(s_h,\vNN_h,\vU_h) \in \Admisuni^h(g_h,\vR_h)$ and the truncations $(I_h\hat{s}_h,I_h\widehat\vU_h)$, then the same assertion holds for the discrete energies.
\end{lemma}
\begin{proof}
We first observe that \eqref{eqn:g_nondegenerate} implies $(\hat{s},\vNN,\widehat{\vU}) \in \Admisuni(g,\vR)$ whereas \eqref{eqn:psi_nondegenerate} yields $\BulkLdG(\hat{s},\vNN)\le \BulkLdG(\hat{s},\vNN)$. Moreover, making use of $|\hat{s}| \le |s|$ and $|\nabla\widehat{s}|\le |\nabla s|$ a.e. in $\Omega$, the inequality $\Eunimain[\hat{s},\vNN] \le \Eunimain[s,\vNN]$ follows immediately.
\end{proof}  

\begin{remark}[range of $s$] For problems in 3d, the admissibility condition $s \in [-1/2, 1]$ is asymmetric with respect to the origin. Since part of our argument below is based on regularizing $|s|$ and afterwards recovering its sign, we need to account for such an asymmetry. A simple way to do so is to consider
\begin{equation} \label{eqn:def_tilde_s}
\check{s} = s_+ - 2 s_-.
\end{equation}
Clearly, the first condition in \eqref{eqn:LdG_structural_conditions} is equivalent to
\begin{equation*} \label{eqn:addmissibility_tilde_s}
-1 \le \check{s} \le 1.
\end{equation*}
In the next result, we consider the regularization using this modified degree of orientation; for simplicity of notation, we drop the ``check'' in $s$.
\end{remark}

\noindent
{\bf Rank-one constraint.}
Our regularization method entails smoothing by convolution. This breaks the uniaxial constraint \eqref{eqn:Q_matrix_uniaxial}, that needs to be rebuilt into the smoothed tensor field; hence, we extract the leading eigenspace. We thus need to account for the dependence of eigenvalues with respect to matrix perturbations. Let $\text{Sym}(d)$ denote the set of symmetric $d\times d$ matrices. Given $\vA \in \text{Sym}(d)$, let $\lambda_1 \ge \ldots \ge \lambda_d$ be the eigenvalues of $\vA$ including multiplicities and $\lambda_{m(1)} > \cdots > \lambda_{m(n)}$ be the $1\le n\le d$ distinct eigenvalues. Let $\{\vP_k\}_{k=1}^n$ be the orthogonal projections onto the eigenspaces associated with $\{\lambda_{m(k)}\}_{k=1}^n$ and let $r(k)\ge1$ be the rank of $\vP_k$; hence $r(k)$ is the multiplicity of $\lambda_{m(k)}$ for $1\le k\le n$. The spectral decomposition of $\vA$ reads
$
\vA = \sum_{k=1}^n \lambda_{m(k)} \vP_k
$
We now consider the set $S^{1,0}(d)$ of non-negative symmetric tensors of rank at most one,
\[
S^{1,0}(d) = \left\{ \vA \in \text{Sym}(d) \colon \vA =  \vu \otimes \vu \ \text{for some } \vu \in \R^d \right\},
\]
and follow \cite{Balan:16} to construct the projection operator $\vPi \colon \text{Sym}(d) \to S^{1,0}(d)$ defined by
\begin{equation} \label{eqn:def_of_P}
\vPi(\vA) = (\lambda_1 - \lambda_2) \vP_1.
\end{equation}
The map $\vPi$ is Lipschitz continuous. This is proven in \cite[Lemma 3.4]{Balan:16} with an explicit Lipschitz constant $3+2^{1+\frac1p}$ (in the $\ell_p$-norm). We give an elementary proof below which relies on the following basic result.
\begin{lemma}[$C^1$ property of $\vPi$] \label{lem:smooth_eigenvalues}
The map $\text{Sym}(d) \to \R^d$, given by $\vA \mapsto (\lambda_1(\vA), \ldots, \lambda_d(\vA))$, is continuous. Moreover, in the set of symmetric matrices whose first eigenspace has dimension $1$
\[
\mbox{Sym}^1(d) := \{ \vA \in  \text{Sym}(d)\colon \lambda_1(\vA) > \lambda_2(\vA) \},
\]
or equivalently the rank of $\vP_1$ is $1$, the map $\vPi$ is of class $C^1$.
\end{lemma}
\begin{proof}
  The eigenvalues $\{\lambda_i(\vA)\}_{i=1}^d$ are the roots of the characteristic polynomial of $\vA$ and depend continuously on the coefficients and so on the entries of $\vA$. To show the $C^1$ property around $\vA_0\in \text{Sym}^1(d)$, let $\vA\in\text{Sym}(d)$ and $\vx_1=\vx_1(\vA)$ be a normalized eigenvector corresponding to the first eigenvalue $\lambda_1=\lambda_1(\vA)$. The equation that defines $(\vx_1,\lambda_1)$ and its derivative with respect to $(\vx_1,\lambda_1)$ read
\[
\vF(\vx_1,\lambda_1,\vA) =
\begin{bmatrix}
  \vA \vx_1 - \lambda_1 \vx_1 \\
  \|\vx_1\|_2^2
\end{bmatrix}
=
\begin{bmatrix}
  0 \\ 1
\end{bmatrix},
\qquad
D_{\vx_1,\lambda_1} \vF(\vx_1,\lambda_1,\vA) = 
\begin{bmatrix}
\vA - \lambda_1 \vI & -\vx_1 \\ 2\vx_1^T & 0
\end{bmatrix}.
\]
Since $\lambda_1(\vA_0)$ is single, the matrix $D_{\vx_1,\lambda_1} \vF(\vx_1(\vA_0),\lambda_1(\vA_0),\vA_0)$ is invertible for otherwise if $(\vy,\alpha)^T\in\R^{d+1}$ is in the kernel it must necessarily vanish. Therefore, the Implicit Function Theorem (IFT) applies thereby giving the existence of $(\vx_1(\vA),\lambda_1(\vA))$ and its $C^1$ dependence on $\vA$; we refer to \cite[Chapter 11.1, Theorem 2]{Evans:book} for a different argument. To prove that $\lambda_2(\vA)$ is also $C^1$ we proceed similarly but note that this eigenvalue might have multiplicity $r(2)>1$. We thus form the equation for $\vP_2=\vP_2(\vA)\in\R^{d\times d}$ being a matrix with rank $r(2)$ and $\lambda_2=\lambda_2(\vA)$
\[
\vF(\vP_2,\lambda_2,\vA) =
\begin{bmatrix}
  \vA \vP_2 - \lambda_2 \vP_2 \\
  \|\vP_2\|_2^2
\end{bmatrix}
=
\begin{bmatrix}
  0 \\ 1
\end{bmatrix},
\qquad
D_{\vP_2,\lambda_2} \vF(\vP_2,\lambda_2,\vA) = 
\begin{bmatrix}
\vA - \lambda_2 \vI & -\vP_2 \\ 2\vP_2^T & 0
\end{bmatrix},
\]
and show that the kernel of this matrix is trivial. The IFT gives the asserted $C^1$ continuity of $\lambda_2(\vA)$.
\end{proof}

\begin{lemma}[Lipschitz property of $\vPi$]\label{L:Lip-rank1}
The map $\vPi \colon \mbox{Sym}(d) \to S^{1,0}(d)$ is uniformly Lipschitz continuous and is invariant on $S^{1,0}(d)$, i.e. $\vPi(\vA)=\vA$ for all $\vA\in S^{1,0}(d)$.
\end{lemma}
\begin{proof}
The invariance of $\vPi$ over $S^{1,0}(d)$ is clear from its definition. Given $\vA,\vB=\vA+\delta\vA\in\mbox{Sym}(d)$, write
\begin{align*}
  \vPi(\vB) - \vPi(\vA) &= \big[ \big(\lambda_1(\vB) - \lambda_1(\vA) \big) - \big( \lambda_2(\vB) - \lambda_2(\vA) \big)  \big] \vP_1(\vB) + 
  \big( \lambda_1(\vA) - \lambda_2(\vA) \big) \big(\vP_1(\vB) - \vP_1(\vA) \big).
\end{align*}
We examine the two terms on the right hand side separately. We split the proof into three steps.

\medskip\noindent
{\it Step 1: Lipschitz property of the first term}.
We resort to Weyl's inequality
for eigenvalues of symmetric matrices \cite[Section III.2]{Bhatia:97}
\[
\big| \lambda_k(\vB) - \lambda_k(\vA) \big| \le \| \vB - \vA \|_2
\quad\forall 1\le k\le d.
\]
Since $\|\vP_1(\vB)\|_2=1$ because $\vP_1(\vB)$ is an orthogonal projection, this proves the Lipschitz property for the first term with constant $2$. If $\lambda_1(\vA)$ is a multiple eigenvalue, then $\lambda_1(\vA)=\lambda_2(\vA)$, the second term vanishes, and the proof is over. We thus assume that $\lambda_1(\vA)$ is simple from now on.

\medskip\noindent
{\it Step 2: Bound on $\|D_\vA \vx_1(\vA;\delta\vA)\|_2$.} In view of  Lemma \ref{lem:smooth_eigenvalues} ($C^1$ property of $\vPi$), we differentiate the equation
$\vF(\vx_1(\vA),\lambda_1(\vA),\vA) = [0,1]^T$ with respect to $\vA$ in the direction $\delta\vA$ to obtain
$
D_{\vx_1} \vF \, D_\vA\vx_1
+ D_{\lambda_1} \vF \, D_\vA\lambda_1
+ D_\vA \vF : \delta\vA = 0
$
where $D_\vA\vx_1 = D_\vA\vx_1(\vA;\delta\vA)$ and $D_\vA\lambda_1 = D_\vA\lambda_1(\vA;\delta\vA)$. Making use of Lemma \ref{lem:smooth_eigenvalues} again, we thus deduce the equation in $\mathbb{R}^{d+1}$
\[
\begin{bmatrix}
\vA - \lambda_1 \vI \\ 2 \vx_1^T
\end{bmatrix}
D_\vA \vx_1 +
\begin{bmatrix}
- \vx_1 \\ 0
\end{bmatrix}
D_\vA \lambda_1 = -
\begin{bmatrix}
\delta\vA \, \vx_1 \\ 0
\end{bmatrix}.
\]
The last row yields $\vx_1^T \, D_\vA \vx_1=0$, whence $D_\vA \vx_1$ is perpendicular to $\vx_1$ and $D_\vA \vx_1 = \sum_{k=2}^d \alpha_k \vx_k$ can be expressed in terms of the orthonormal eigenvectors $\{\vx_k\}_{k=1}^d$ of $\vA$ without component along $\vx_1$. Moreover, if $\delta\vA \, \vx_1 = \sum_{k=1}^d \beta_k\vx_k$, then the first $d$ rows of the preceding equation give
\[
\sum_{k=2}^d \big[ \alpha_k (\lambda_k - \lambda_1) + \beta_k \big] \vx_k
= \big( D_\vA \lambda_1 - \beta_1 \big) \vx_1.
\]
This obviously implies $D_\vA \lambda_1 = \beta_1$ and
\[
\alpha_k = \frac{\beta_k}{\lambda_1-\lambda_k} \quad\forall \, 2\le k \le d.
\]
Let $\bm{\alpha} = (\alpha_k)_{k=1}^d$ with $\alpha_1=0$ and $\bm{\beta} = (\beta_k)_{k=1}^d$. Since $\|\bm{\beta}\|_2 \le \|\delta\vA\|_2$, we see that
\[
\|D_\vA\vx_1(\vA;\delta\vA)\|_2 = \|\bm{\alpha}\|_2 \le \frac{ \|\delta\vA\|_2 }{\lambda_1(\vA) - \lambda_2(\vA)},
\]
because $0<\lambda_1(\vA) - \lambda_2(\vA) \le \lambda_1(\vA) - \lambda_k(\vA)$ for all
$2\le k \le d$.

\medskip\noindent
{\it Step 3: Lipschitz property of the second term.} Exploiting that $\vP_1(\vA) = \vx_1(\vA)\otimes\vx_1(\vA)$, we readily get
\[
D_\vA \vP_1(\vA;\delta\vA) = D_\vA \vx_1(\vA;\delta\vA) \otimes \vx_1(\vA)
+ \vx_1(\vA) \otimes D_\vA \vx_1(\vA;\delta\vA).
\]
Since $\vx_1(\vA)$ and $D_\vA \vx_1(\vA;\delta\vA)$ are perpendicular, we infer that
\begin{equation} \label{eq:orthogonal_bound}
\| D_\vA \vP_1(\vA;\delta\vA) \|_2 \le \|D_\vA \vx_1(\vA;\delta\vA)\|_2. 
\end{equation}  
Indeed, if $\vu,\vv \in \R^d$ are orthonormal and $\vw \in \R^d$, then 
$
(\vu \otimes \vv + \vv \otimes \vu) \vw = (\vv \cdot \vw) \vu + (\vu \cdot \vw) \vv,
$
and thus, by Bessel's inequality, 
\[
\| (\vu \otimes \vv + \vv \otimes \vu) \vw \|_2 \le \| \vw \|_2;
\]
estimate \eqref{eq:orthogonal_bound} then follows by scaling.
Combining this with Step 2 gives
\[
\big| \lambda_1(\vA) - \lambda_2(\vA) \big| \, \|\vP_1(\vA+\delta\vA) - \vP_1(\vA)\|_2
\le \|\delta\vA\|_2 + \big| \lambda_1(\vA) - \lambda_2(\vA) \big| \, o(\|\delta\vA\|_2)
= \big( 1 + o(1) \big) \|\delta\vA\|_2, 
\]
which shows that the desired Lipschitz constant is $1$. Altogether the uniform Lipschitz constant of $\vPi$ (with respect to the $\ell_2$-norm) is $3$. This concludes the proof.

\end{proof}

\noindent
{\bf Regularization.}
We now have all the tools we need to prove that Lipschitz continuous functions are dense in the Landau - de Gennes restricted admissible class $ \Admisuni(g,\vR)$.

\begin{prop}[regularization] \label{prop:LdG_regularization}
Let \eqref{eqn:g_geq_0}, \eqref{eqn:psi_nondegenerate} and \eqref{eqn:def_tilde_s} hold. Given $\eps > 0$ and $(s,\vNN,\vU) \in \Admisuni(g,\vR)$ with 
\begin{equation}\label{eqn:bounds_g}
-1 + \delta_0 \le s \le 1 -\delta_0 \quad\text{a.e. } \Om
\end{equation}
there exists a sequence $(s_\eps,\vNN_\eps,\vU_\eps) \in \Admisuni(g,\vR)$ such that $(s_\eps,\vU_\eps) \in W^{1,\infty}(\Om) \times [W^{1,\infty}(\Om)]^{d\times d}$, and
\begin{equation}\label{eqn:LdG_bound_s_eps}
  \begin{gathered}
  \| (s,\vU) - (s_\eps, \vU_\eps) \|_{H^1(\Om)} < \eps,
  \qquad
  \| \vNN - \vNN_\eps \|_{L^2(\Omega\setminus\Sing)} < \eps,
  \\
  -1 + \delta_0 \le s_\eps \le 1 -\delta_0.
  \end{gathered}
\end{equation}
\end{prop}
\begin{proof} We proceed in several steps.

\medskip\noindent
{\em Step 1: Regularization with boundary condition.} 
Consider a zero-extension of $s-g \in H^1_0(\Omega)$ over $\R^d\setminus\Omega$. Given $\delta > 0$, we set
\[
\omega_\delta := \{ x \in \Omega \colon d(x, \dOm) \le \delta \},
\]
and define $d_\delta (x) = \chi_\Omega (x) \min \left\{ \frac{1}{\delta} d(x, \dOm), 1 \right\}$,  which is a Lipschitz continuous function, with ${\rm supp}(\nabla d_\delta) \subset \omega_\delta$ and $| \nabla d_\delta | = \delta^{-1} \chi_{\omega_\delta}$. Let $\eta_\delta$ be a smooth, nonnegative mollifier supported in $B_\delta(0)$, and define
\[ \begin{aligned}
& s_\delta := d_\delta (s \ast \eta_\delta) + (1 - d_\delta) g , \\
& \vwU_\delta := d_\delta \left( \vwU \ast \eta_\delta \right) + (1 - d_\delta)\vR.
\end{aligned} \]
where $\vwU := \sign (s) \vU = |s| \vNN \in [H^1(\Om)]^{d\times d}$ coincides with $\vR$ on $\dOm$ (because of \eqref{eqn:g_geq_0}). We thus have $(s_\delta, \vwU_\delta) = (s, \vR)$ on $\dOm$ and arguing as in \cite[Proposition 3.2, Step 1]{Nochetto_SJNA2017} it follows that
\begin{equation*}
  s_\delta \to s, \quad \vwU_\delta \to \vwU \quad  \text{a.e. and in } H^1(\Om).
\end{equation*}
The choice to regularize the field $\vwU$ instead of $\vU$ is motivated by the next step. Since convolution breaks the uniaxial structure of tensor fields, we cannot preserve the trace condition $s = \tr[\vU]$. However, convolution does preserve positive-semidefiniteness, which is a property that $\vwU$ satisfies. Additionally, we shall recover the rank-one constraint by means of the map $\vPi$ defined in \eqref{eqn:def_of_P}. Because $\vwU \in S^{1,0}(d)$, we have $\vPi(\vwU) = \vwU$; in contrast, if $s<0$, we have $\vPi(\vU) = \mathbf{0}$ when $d>2$ and $\vPi(\vU) = - s\vNN^\bot$ when $d=2$, where $\vNN^\bot$ is the line field orthogonal to $\vNN$ a.e. in $\Om$.

\medskip\noindent
{\em Step 2: Preserve structural conditions.}
We now rebuild these conditions into the regularized pair $(s_\delta, \vwU_\delta)$ by introducing a coarser scale.
Our assumption \eqref{eqn:bounds_g} implies that the extension of $s$ satisfies the same bound on $\R^d \setminus \Om$. Therefore, we also have 
$-1 + \delta_0 \le s_\delta(x) \le 1-\delta_0$ on $\R^d$.
Moreover, we have $\lambda_1(\vwU_\delta) \le 1-\delta_0$ since, given any vector $\vv \in \R^d$, with $|\vv|=1$, there holds for $\delta$ sufficiently small that
\[
|\vwU_\delta \vv \cdot \vv | \le d_\delta |\sign (s) \vU \vv \cdot \vv \ast \eta_\delta| + (1 - d_\delta) |\vR \vv \cdot \vv | \le  1- \delta_0,
\]
because $|\lambda_1(\vU)| \le 1-\delta_0$ a.e. in $\Om$ and $|\lambda_1(\vR)| \le 1-\delta_0$ in a neighborhood of $\dOm$.
\begin{center}
\begin{figure}[ht]
\includegraphics[width=2.in]{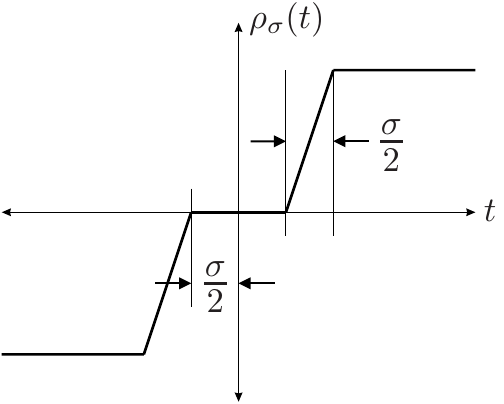}
\caption{Regularized sign function.}
\label{fig:regularized_sign_func}
\end{figure}
\end{center}
We introduce a parameter $\sigma > \delta$ and the following regularization of the sign function (see Figure \ref{fig:regularized_sign_func}):
\[
\rho_\sigma (t) = \begin{cases}
\sign (t) & \text{ if } \sigma< |t|, \\
\frac{2 \, \sign(t) }{\sigma} (|t| - \sigma / 2)  & \text{ if } \sigma/2 < |t| \le \sigma, \\
0 & \text{ if } |t| \le\sigma/2.
\end{cases}
\]
An elementary verification gives
\[
\rho_\sigma (s_\delta) \to \rho_\sigma (s) \text{ as } \delta \to 0, \quad \quad  \text{a.e. and in } H^1(\Om).
\]
Next, we use the operator $\vPi$ given by \eqref{eqn:def_of_P} to define
\[ \begin{aligned}
& s_{\sigma, \delta} := \rho_\sigma(s_\delta) \, \tr [ \vPi(\vwU_\delta)] =  \rho_\sigma(s_\delta) | \vPi(\vwU_\delta) |, \\
& \vU_{\sigma, \delta} :=  \rho_\sigma(s_\delta) \vPi(\vwU_\delta).
\end{aligned} \]
Since $\tr [ \vPi(\vwU_\delta)] = \lambda_1 (\vPi(\vwU_\delta)) \in [0,1-\delta_0]$ and $-1 \le \rho_\sigma \le 1$, we deduce that $-1+\delta_0 \le s_{\sigma,\delta} \le 1-\delta_0$; thus, we have $\vU_{\sigma, \delta} = s_{\sigma, \delta} \vNN_{\sigma, \delta}$ for some $\vNN_{\sigma, \delta} \in \LL^{d-1}$ and $(s_{\sigma,\delta},\vNN_{\sigma,\delta},\vwU_{\sigma,\delta})$ satisfies the structural condition \eqref{eqn:LdG_structural_conditions}.

Under assumption \eqref{eqn:g_geq_0}, it follows that if $\sigma < \delta_0$ then  $s_\delta = g > \sigma$ on $\dOm$, so that $\rho_\sigma(s_\delta) = 1$ on $\dOm$. Thus,
\[ \begin{aligned}
& s_{\sigma,\delta} = \tr [\vPi(\vwU_\delta)] = \tr (\vR) = g \quad \mbox{on } \dOm, \\
& \vU_{\sigma,\delta} = \vPi(\vwU_\delta) = \vR \quad \mbox{on } \dOm.
\end{aligned} \]
Therefore, $(s_{\sigma,\delta}, \vNN_{\sigma,\delta}, \vU_{\sigma,\delta} ) \in \Admisuni(g,\vR,\vM)$. We still need to choose $\sigma$ and $\delta$ such that $(s_{\sigma,\delta}, \vU_{\sigma,\delta})$ is sufficiently close to $(s,\vU)$ in $[H^1(\Om)]^{1+d\times d}$.

\medskip\noindent
    {\em Step 3: Convergence as $\delta \to 0$.} Since $\vPi$ is Lipschitz in view of Lemma \ref{L:Lip-rank1}, it is immediate to see that
\[ \left\lbrace \begin{array}{l}
s_{\sigma,\delta} \to s_\sigma := \rho_\sigma (s) \tr[ \vPi(\vwU)] = \rho_\sigma (s) \tr[\vwU] \\
\vU_{\sigma,\delta} \to \vU_\sigma := \rho_\sigma (s) \vPi(\vwU) = \rho_\sigma (s) \vwU \\
\end{array} \right.  
\quad \mbox{a.e. and in } L^2(\Om),
\]
as $\delta \to 0$. Consider now the set $\Lambda_\sigma:=\{|s| > \frac{\sigma}{2}\}$ to deal with $\vNN_{\sigma,\delta}$. The fact that $s_\delta \to s, \vwU_\delta\to\vwU$ a.e. yields $\rho_\sigma(s_\delta(x))\ne0$, $\tr[\vPi(\vwU_\delta(x))] \ne 0$ for a.e. $x\in\Lambda_\sigma$ provided $\delta$ is sufficiently small depending on $x$. Hence
\[  
\vNN_{\sigma,\delta} = \frac{\vU_{\sigma,\delta}}{s_{\sigma,\delta}} =
\frac{\vPi(\vwU_\delta)}{\tr[\vPi(\vwU_\delta)]} \to
\frac{\vwU}{\tr[\vwU]} = \frac{\vwU}{|s|} = \vNN
\quad\text{a.e. in $\Lambda_\sigma$ and in } L^2(\Lambda_\sigma) \text{, as } \delta \to 0.
\]

We next prove convergence in $H^1(\Om)$.
For $i,j = 1, \ldots, d$, we have
\begin{equation} \label{eqn:convergence_delta} \left\{ \begin{aligned} 
& \nabla [(\vU_{\sigma,\delta})_{ij}] = \rho'_\sigma (s_\delta) \nabla s_\delta \vPi(\vwU_\delta)_{ij} + \rho_\sigma (s_\delta) \nabla  [\vPi(\vwU_\delta)_{ij}], \\
& \nabla [(\vU_{\sigma})_{ij}] = \rho'_\sigma (s) \nabla s \vPi(\vwU)_{ij} + \rho_\sigma (s) \nabla  [\vPi(\vwU)_{ij}].
\end{aligned} \right. \end{equation}
It suffices to check convergence term by term in the right hand sides in \eqref{eqn:convergence_delta}. For the first one, we write
\[ \begin{aligned}
\rho'_\sigma (s_\delta) \nabla s_\delta \vPi(\vwU_\delta)_{ij} -  \rho'_\sigma (s) \nabla s \vPi(\vwU)_{ij} & = \nabla (s_\delta - s) \rho'_\sigma (s_\delta) \vPi(\vwU_\delta)_{ij} \\
& \quad +  \nabla s \left[ \rho'_\sigma (s_\delta) \vPi(\vwU_\delta)_{ij} - \rho'_\sigma (s) \vPi(\vwU)_{ij} \right].
\end{aligned} \]
Since $\nabla (s_\delta - s) \to 0$ in $L^2(\Omega)$ and $|\rho'_\sigma (s_\delta) \vPi(\vwU_\delta)_{ij}|$ is bounded, we deduce that 
\[
\iO |\nabla (s_\delta - s)|^2 \big|\rho'_\sigma (s_\delta) \vPi(\vwU_\delta)_{ij}\big|^2 dx \to 0.
\]
As for the remaining term, we write
\[
\rho'_\sigma (s_\delta) \vPi(\vwU_\delta)_{ij} - \rho'_\sigma (s) \vPi(\vwU)_{ij} = 
[\rho'_\sigma (s_\delta) - \rho'_\sigma (s) ] \vPi(\vwU_\delta)_{ij} + \rho'_\sigma (s) [\vPi(\vwU_\delta)_{ij} - \vP(\vwU)_{ij}]
\]
and notice that 
\[ \begin{aligned}
& \rho'_\sigma (s_\delta) - \rho'_\sigma (s) \to 0 \quad \mbox{in } L^2(\Omega), \\
& \vPi(\vwU_\delta)_{ij} \quad \mbox{remains bounded,} \\
& | \vPi(\vwU_\delta)_{ij} - \vPi(\vwU)_{ij} | \le | \vPi(\vwU_\delta) - \vPi(\vwU) | \le C | \vwU_\delta - \vwU | \to 0  \quad \mbox{in } L^2(\Omega),
\end{aligned} \]
according to Lemma \ref{L:Lip-rank1}.
This shows convergence of the first terms in the right hand sides in \eqref{eqn:convergence_delta}:
\[
 \rho'_\sigma (s_\delta) \nabla s_\delta \vPi(\vwU_\delta) \to  \rho'_\sigma (s) \nabla s \vPi(\vwU) \quad \mbox{in } L^2(\Omega).
\]

To prove that
$
\rho_\sigma (s_\delta) \nabla  [\vPi(\vwU_\delta)] \to \rho_\sigma (s) \nabla  [\vPi(\vwU)] \mbox{ in } L^2(\Omega),
$
we write
\begin{equation} \label{eqn:splitting_regularization}
\begin{aligned}
\rho_\sigma (s_\delta) \nabla  [\vPi(\vwU_\delta)]  - \rho_\sigma (s) \nabla  [\vPi(\vwU)]
= & (\rho_\sigma (s_\delta) - \rho_\sigma (s)) \nabla  [\vPi(\vwU)] \\
& + \rho_\sigma(s_\delta) D\vPi (\vwU_\delta) \nabla(\vwU_\delta - \vwU) + \rho_\sigma(s_\delta) (D\vPi (\vwU_\delta) - D\vPi (\vwU) ) \nabla\vwU.
\end{aligned}
\end{equation}
The first term in the right hand side above converges to $0$ in $L^2(\Om)$ because $\nabla  [\vPi(\vwU)_{ij}] \in L^2(\Omega)$ and $|\rho_\sigma (s_\delta) - \rho_\sigma (s)|$ is bounded and converges to $0$ a.e. in $\Omega$. As for the second term in \eqref{eqn:splitting_regularization}, we use Lemma \ref{L:Lip-rank1} (Lipschitz property of $\vPi$) and the boundedness of $\rho_\sigma$ to obtain
\[
\int_\Omega \rho^2_\sigma(s_\delta) |D\vPi (\vwU_\delta)|^2 |\nabla(\vwU_\delta - \vwU)|^2 \le \|D\vPi\|_\infty^2 \int_\Omega |\nabla(\vwU_\delta - \vwU)|^2 \to 0,
\]
because $\vwU_\delta \to \vwU$ in $H^1(\Omega)$.

Finally, to prove that the last term in \eqref{eqn:splitting_regularization} converges to $0$ in $L^2(\Om)$, we consider $\Lambda_\sigma$ as above, namely
\[
\Lambda_\sigma = \{ |s| > \sigma/2 \}, \quad \Om\setminus\Lambda_\sigma = \{ |s| \le \sigma/2 \}.
\]
In the region $\Omega\setminus\Lambda_\sigma$, we have $\rho_\sigma (s_\delta) \to \rho_\sigma(s) = 0$ a.e.. Using this together with the boundedness of $|\rho_\sigma (s_\delta)|$ and $|D\vPi|$, and the fact that $\nabla \vwU \in L^2(\Om)$, we obtain
\[
\int_{\Omega\setminus\Lambda_\sigma} |\rho_\sigma (s_\delta)|^2  |D\vPi (\vwU_\delta) - D\vPi (\vwU)|^2 |\nabla \vwU|^2 \to 0.
\] 
On the other hand, we have that for a.e. $x \in \Lambda_\sigma$, $\vwU(x) = |s(x)| \vNN(x) \in \mbox{Sym}^{1}(d).$ Also, since $\vwU_\delta \to \vwU$ and $\lambda_1(\vwU(x)) = |s(x)| \ge \sigma/2$ a.e. $x \in \Lambda_\sigma$, there exists a $\delta'$ (depending on $x$) such that $\vwU_\delta(x) \in \mbox{Sym}^{1}(d)$ for all $\delta \le \delta'$. Using that $\vPi$ is of class $C^1$ in $\mbox{Sym}^{1}(d)$, according to Lemma \ref{lem:smooth_eigenvalues}, we deduce that
\[
D \vPi(\vwU_\delta) \to D \vPi(\vwU) \quad \mbox{a.e. in } \Lambda_\sigma. 
\]
Therefore, applying again the Dominated Convergence Theorem yields
\[
\int_{\Lambda_\sigma} |\rho_\sigma (s_\delta)|^2  |D\vPi (\vwU_\delta) - D\vPi (\vwU)|^2 |\nabla \vwU|^2 \to 0.
\]

We have thus proved that
\[ \left\lbrace \begin{array}{l}
s_{\sigma,\delta} \to s_\sigma := \rho_\sigma (s) \tr(\vwU) \\
\vU_{\sigma,\delta} \to \vU_\sigma := \rho_\sigma (s) \vwU \\
\end{array} \right.  
\quad \mbox{in } H^1(\Om), \mbox{ as } \delta \to 0.
\]

\medskip\noindent
{\em Step 4: Convergence as $\sigma \to 0$.}
Because  $\vwU = |s| \vNN$, a straightforward calculation gives
\[ \left\lbrace \begin{array}{l}
s_\sigma = \rho_\sigma (s) \tr(\vwU) \to s \\
\vU_\sigma = \rho_\sigma (s) \vwU \to \vU \\
\end{array} \right.
\quad \mbox{a.e. and in } L^2(\Om), \mbox{ as }  \sigma \to 0.
\]
To prove convergence in $H^1(\Om)$ we observe that $\vU_\sigma=\rho_\sigma(s) \vwU = \rho_\sigma(s) \, \sign(s) \, \vU = |\rho_\sigma(s)| \, \vU$, whence
\[
\nabla(\vU_\sigma - \vU) = \nabla\big[\big( |\rho_\sigma(s)| - 1 \big) \vU  \big]
= \nabla |\rho_\sigma(s)| \, \vU + \big( |\rho_\sigma(s)| - 1 \big) \nabla \vU
\]
We show that these two terms tend to zero separately in $L^2(\Om)$. First note that
\[
\big|\nabla |\rho_\sigma(s)|\big| = \rho'_\sigma(s) \big|\nabla s\big| = \frac{2}{\sigma} \chi_{\{\frac{\sigma}{2}<|s|<\sigma\}} \big|\nabla s\big| 
\]
whereas $|\vU| = |s| < \sigma$ in the set $\{\frac{\sigma}{2}<|s|<\sigma\}$. Since
$\chi_{\{\frac{\sigma}{2}<|s|<\sigma\}}\to0$ a.e. in $\Om$ as $\sigma\to0$, and $\big|\nabla s\big| \in L^2(\Om)$, we infer from the Dominated Convergence Theorem that
\[
\int_\Om \big| \nabla |\rho_\sigma(s)| \, \vU  \big|^2 \to 0
\quad\text{as } \sigma\to0.
\]
On the other hand, in view of the definition of $\rho_\sigma(s)$, we have
\[
\int_\Om \big| \big( |\rho_\sigma(s)| - 1 \big) \nabla \vU \big|^2
\le \int_\Om \chi_{\{|s| \le \sigma \}} \big|\nabla \vU \big|^2
= \int_\Om \chi_{\{|\vU| \le \sigma \}} \big|\nabla \vU \big|^2
\to  \int_\Om \chi_{\{|\vU| = 0 \}} \big|\nabla \vU \big|^2 = 0
\]
because $\nabla v = 0$ a.e. in $\{v=0\}$ for any $v \in H^1(\Om)$ \cite[Ch. 5, Exercise 17]{Evans:book}. We have thus proved that $\nabla(\vU_\sigma-\vU)\to0$ in $L^2(\Om)$ as $\sigma\to0$.

It remains to deal with $s_\sigma-s$. We write $s_\sigma = \rho_\sigma(s) \, \tr(\sign(s) \, \vU) = |\rho_\sigma(s)| \, \tr(\vU)$ to realize that
\[
\nabla(s_\sigma-s) = \nabla\big[\big( |\rho_\sigma(s)| - 1 \big) \, \tr (\vU)\big] =
\nabla |\rho_\sigma(s)| \, \tr (\vU) + \big( |\rho_\sigma(s)| - 1 \big) |\nabla \tr (\vU)|.
\]
This expression has the same structure as $\nabla(\vU_\sigma-\vU)$ except that $\vU$ is now replaced by $\tr (\vU)$. Therefore, the same argument as before yields as $\sigma\to0$
\[
\nabla(s_\sigma-s) \to 0 \quad \text{in } L^2(\Om).
\]

\medskip\noindent
{\em Step 5: Choice of $\sigma$ and $\delta$.} 
Given $\eps > 0$, we first choose $\sigma  > 0$ such that
\[ 
 \| \vU_\sigma - \vU \|_{H^1(\Omega)} \le \eps/2, \quad
 \| s_\sigma - s \|_{H^1(\Omega)} \le \eps/2, \quad
 \|\vNN - \vNN \chi_{\{|s| > \frac{\sigma}{2}\}} \|_{L^2(\Om\setminus\Sing)} \le \eps/2,
\]
because $\chi_{\{ |s| > \frac{\sigma}{2} \}} \to \chi_{\{ |s| > 0\}}$ a.e. as $\sigma\to0$ and
$\Om\setminus\Sing = \{|s|>0\}$. Since $\chi_{\{0<|s|\le\frac{\sigma}{2}\}} \to 0$ a.e. and $|\vNN_{\sigma,\delta}|=1$, we can further reduce $\sigma$ so that
\[
\|\vNN_{\sigma,\delta}\|_{L^2(\{0<|s|\le\frac{\sigma}{2}\})} \le \eps/4.
\]
Finally, take $\delta \le \sigma$ such that
\[ 
\| \vU_{\sigma, \delta} - \vU_\sigma \|_{H^1(\Omega)} \le \eps/2, \quad
\| s_{\sigma, \delta} - s_\sigma \|_{H^1(\Omega)} \le \eps/2, \quad
\| \vNN_{\sigma, \delta} - \vNN \|_{L^2(\{|s| > \frac{\sigma}{2} \})} \le \eps/4.
\]
The proof concludes upon defining $(s_\eps,\vNN_\eps, \vU_\eps) := (s_{\sigma,\delta}, \vNN_{\sigma,\delta}, \vU_{\sigma,\delta})$.
\end{proof}

With this regularization result at hand, we now address the construction of a recovery sequence. Given $\eps>0$, let $(s_{\eps,h},\vU_{\eps,h}) := \big(I_h(s_{\eps,h}),I_h(\vU_{\eps,h})\big)$ be the Lagrange interpolants of the regularized pair $(s_\eps, \vU_\eps)$ constructed in Proposition \ref{prop:LdG_regularization}, that are well-defined because $(s_\eps,\vU_\eps) \in W^{1,\infty}(\Om) \times [W^{1,\infty}(\Om)]^{d\times d}$. We define the line field $\vNN_{\eps,h} \in \THh$ so that, at the node $x_i \in \Nk_h$ it satisfies
\[
\vNN_{\eps,h} (x_i) = \left\lbrace\begin{array}{cl}
\vU_\eps(x_i) / s_\eps(x_i) & \mbox{ if } s_{\eps}(x_i) \ne 0, \\
\mbox{any tensor in } \LL^{d-1} & \mbox{ if } s_{\eps}(x_i) = 0.
\end{array} \right.
\]
This definition guarantees that $\vU_{\eps,h} = I_h(s_{\eps,h}\vNN_{\eps,h})$, whence the structural condition \eqref{eqn:LdG_struct_condition_discrete} is satisfied and thus $(s_{\eps,h}, \vNN_{\eps,h},\vU_{\eps,h}) \in \Admisuni^h(g_h, \vR_h)$. Because $(s_{\eps,h}, \vU_{\eps,h}) \to (s_\eps, \vU_\eps)$ in $H^1(\Om)\times[H^1(\Om)]^{d\times d}$ as $h\to0$, we readily deduce that \eqref{eqn:LdG_convergence_H1} is satisfied. Proving \eqref{eqn:LdG_limsup} is equivalent to showing that $\mathcal{E}^h \to 0$, the consistency term in \eqref{eqn:LdG_residual}, and can be done using the same arguments as in \cite[Lemma 3.3]{Nochetto_SJNA2017}. We omit the proof.

\begin{lemma}[lim-sup inequality] \label{lem:LdG_limsup}
Let $(s_\eps,\vNN_\eps,\vU_\eps)\in \Admisuni(g,\vR)$ be the functions constructed in Proposition \ref{prop:LdG_regularization} and  $(s_{\eps,h}, \vNN_{\eps,h},\vU_{\eps,h}) \in \Admisuni^h(g_h, \vR_h)$ be the discrete functions defined above. Then,
\[
\Eunimain[s_\eps,\vNN_\eps] = \lim_{h \to 0} \Eunimain^h[s_{\eps,h} \vNN_{\eps,h}] = \lim_{h \to 0} \EuniUmain^h[s_{\eps,h} \vU_{\eps,h}] = \EuniUmain[s_\eps,\vU_\eps] .
\]
\end{lemma}

\subsection{Lim-inf property: Weak lower semicontinuity}

This property hinges on convexity of the underlying functional. However, this is not apparent for the main energy in \eqref{eqn:Q_tensor_energy_sU}
\[
\EuniUmain[\widetilde{s},\widetilde{\vU}] = -\frac{1}{2d} \iO  |\nabla \widetilde{s}|^2 \, dx + \frac12 \iO |\nabla \widetilde{\vU}|^2 \, dx.
\]  
because of the negative sign. What restores convexity is the structural property \eqref{eqn:LdG_structural_conditions}, which reads $\widetilde{\vU} = \widetilde{s} \vNN$ in terms of the triple $(\widetilde{s},\vNN,\widetilde{\vU})$, along with $|\widetilde{\vU}| = |\widetilde{s}|$ and equalities
\[
\big|\nabla\widetilde{s}\big| = \big| \nabla |\widetilde{s}| \big| = \big| \nabla |\widetilde{\vU}| \big| = \big| \nabla \widetilde{\vU} \big| \qquad\text{a.e. }\Om.
\]
This reveals the fundamental convexity property of $\EuniUmain[\widetilde{s},\widetilde{\vU}]$, namely
\[
\EuniUmain[\widetilde{s},\widetilde{\vU}] = \frac{d-1}{2d} \int_\Om \big| \nabla |\widetilde{\vU}| \big|^2 dx.
\]

The discretization poses a severe challenge to convexity because the discrete variables $(\widetilde{s}_h,\widetilde{\vU}_h)$ defined in \eqref{eqn:def_of_tilde} satisfy $|\widetilde{s}_h| = |\widetilde{\vU}_h|$ only at the mesh nodes and $\nabla\widetilde{s}_h \ne \nabla |s_h|$. However, upon flattening the matrix $\vU_h$ into a vector and exploiting that the Euclidean norm of the gradient of the flattened matrix coincides with the Fr\"{o}benius norm $|\nabla\vU_h|$, we resort to \cite[Lemma 3.4]{Nochetto_SJNA2017} to establish the following result.
\begin{lemma}[weak lower semi-continuity] \label{lem:LdG_wlsc}
If $\vW_h\in\Uh$ converges weakly in $[H^1(\Omega)]^{d \times d}$ to $\vW$, then
\[
\liminf_{h \to 0} \left(-\frac1d \iO  | \nabla I_h |\tr (\vW_h)| |^2 + \iO | \nabla \vW_h|^2
\right) \ge -\frac{1}{d} \iO | \nabla |\tr (\vW)| |^2 + \iO | \nabla \vW|^2.
\]
\end{lemma}

\subsection{Equicoercivity and compactness}

The last ingredient to prove the convergence of minimum problems is some form of compactness. This follows by deriving uniform bounds in $H^1$ for the discrete minimizers $(s_h, \vU_h)$ and $(\widetilde{s}_h, \widetilde{\vU}_h) = (I_h|s_h|, I_h (|s_h| \vNN_h) )$.

\begin{lemma}[coercivity] \label{lem:LdG_coercivity}
Given $(s_h, \vNN_h, \vU_h) \in \Admisuni^h(g_h, \vR_h)$, we have
\begin{equation} \label{eqn:LdG_coercivity}
\Eunimain^h[s_h,\vNN_h] \ge \frac{d-1}{2d} \max \left\{ \| \nabla \vU_h \|_{L^2(\Om)}^2, \| \nabla s_h \|_{L^2(\Om)}^2 \right\},
\end{equation}
and
\begin{equation} \label{eqn:LdG_coercivity_tilde}
\Eunimain^h[s_h,\vNN_h] \ge \frac{d-1}{2d} \max \left\{ \| \nabla \widetilde{\vU}_h \|^2_{L^2(\Om)}, \| \nabla \widetilde{s}_h \|_{L^2(\Om)}^2 \right\}.
\end{equation}
\end{lemma}
\begin{proof}
First of all, definition \eqref{eqn:LdG_discrete_energy_main} of $\Eunimain^h$ in conjunction with \eqref{eqn:dirichlet_integral_identity} and \eqref{eqn:weakly-acute} readily yields
\[
\Eunimain^h[s_h,\vNN_h] \ge \frac{d-1}{4d} \sum_{i,j = 1}^n k_{ij} (\dij s_h)^2 = \frac{d-1}{2d} \| \nabla s_h \|_{L^2(\Om)}^2.
\]
Moreover, because $|\dij \widetilde{s}_h| \le | \dij s_h|$ for all $i,j = 1, \ldots, n$, we also have $ \| \nabla \widetilde{s}_h \|_{L^2(\Om)} \le  \| \nabla s_h \|_{L^2(\Om)}$.

Secondly, combining \eqref{eqn:EuniUmain} and \eqref{eqn:energyequality} with $\consistLdG_h\ge0$, we obtain
\[ 
\frac12 \| \nabla \vU_h \|^2_{L^2(\Om)} = \Eunimain^h[s_h, \vNN_h]
+ \frac{1}{2d} \|\nabla s_h\|_{L^2(\Omega)}^2 - \consistLdG_h
\le \frac{d}{d-1} \Eunimain^h[s_h, \vNN_h].
\]
Estimate $\frac{d-1}{2d} \| \nabla \widetilde{\vU}_h \|^2_{L^2(\Om)} \le \Eunimain^h[s_h, \vNN_h]$ follows similarly from \eqref{eqn:energyequality_tilde}.
\end{proof}

Our next goal is to show that, from sequences of discrete functions $(s_h,\vNN_h,\vU_h)$ and $(\widetilde{s}_h,\vNN_h,\widetilde{\vU}_h)$ with uniformly bounded energies, it is possible to extract subsequences that converge to admissible functions. For that purpose, we need an elementary auxiliary result.

\begin{lemma}[admissible tensors] \label{lem:rank1_characterization}
Let $\vM \in \text{Sym}(d)$ be such that $\tr(\vM^k) = [\tr(\vM)]^k$ for all $k = 1, \ldots , d$. Then, at least $d-1$ eigenvalues of $\vM$ are equal to zero, i.e., $\vM$ has rank less than or equal to 1.
\end{lemma}

We are now ready to pursue our goal. The key point in the next result is to verify that the candidate tensor fields satisfy the rank-one constraint.

\begin{lemma}[characterization of limits] \label{lem:LdG_limits}
Let a sequence $(s_h, \vNN_h, \vU_h) \in \Admisuni^h(g_h, \vR_h)$ satisfy
\[
\Eunimain^h[s_h, \vNN_h] \le \Lambda \quad \forall h > 0,
\]
for some constant $\Lambda$ independent of $h$, and let $\widetilde{s}_h = I_h (|s_h|)$, $\widetilde{\vU}_h = I_h (|s_h| \vNN_h)$ as in \eqref{eqn:def_of_tilde}.
Then, there exist subsequences (not relabeled) $(s_h, \vU_h) \in \X_h$ and $(\widetilde{s}_h, \widetilde{\vU}_h) \in \X_h$, and functions
$(s, \vU), (\widetilde{s}, \widetilde{\vU})\in H^1(\Om)\times[H^1(\Om)]^{d \times d}$ and $\vNN \in L^\infty(\Om;\LL^{d-1})$ such that:
\begin{itemize}
\item $(s_h, \vU_h) \to (s, \vU)$ in $L^2(\Om)\times[L^2(\Om)]^{d \times d}$,  a.e in $\Om$,  $(s_h, \vU_h) \rightharpoonup (s, \vU)$ in $H^1(\Om) \times [H^1(\Om)]^{d \times d}$;
\item $(\widetilde{s}_h, \widetilde{\vU}_h) \to (\widetilde{s}, \widetilde{\vU})$ in  $L^2(\Om)\times[L^2(\Om)]^{d \times d}$,  a.e in $\Om$, $(\widetilde{s}_h, \widetilde{\vU}_h) \rightharpoonup (\widetilde{s}, \widetilde{\vU})$ in $H^1(\Om) \times [H^1(\Om)]^{d \times d}$;
\item the limits satisfy $\widetilde{s} = |s| = \tr [\widetilde{\vU}]$, $s = \tr [\vU]$,  a.e. in $\Om$;
\item $\vNN_h \to \vNN$ a.e. in $\Om \setminus \Sing$, and in $L^2(\Om\setminus \Sing)$, and $\vU = s \vNN$, $\widetilde \vU = \widetilde s \vNN$ a.e. in $\Om$;
\item $\vNN$ admits Lebesgue gradient $\nabla\vNN$ a.e. in $\Om \setminus \Sing$ and
$|\nabla\widetilde{\vU}|^2 = |\nabla \widetilde{s}|^2 + \widetilde{s}^2 |\nabla\vNN|^2$ is valid a.e. in $\Omega\setminus\Sing$;
\end{itemize}
where $\LL^{d-1}$ is defined in \eqref{eqn:line-fields} and $\Sing$ in \eqref{eqn:singular_set}.
\end{lemma}
\begin{proof}
Because the discrete energy $\Eunimain^h[s_h, \vNN_h]$ is uniformly bounded, Lemma \ref{lem:LdG_coercivity} guarantees that the sequences $(s_h, \vU_h)$ and $(\widetilde{s}_h, \widetilde{\vU}_h)$ are bounded in $H^1(\Om) \times [H^1(\Om)]^{d\times d}$. Thus,
we can extract subsequences (not relabeled) such that 
\[
(s_h, \vU_h) \to (s, \vU) \quad \mbox{and} \quad
(\widetilde{s}_h, \widetilde{\vU}_h) \to (\widetilde{s}, \widetilde{\vU}),
\]
strongly in $L^2(\Om)\times[L^2(\Om)]^{d\times d}$, a.e. in $\Om$, and weakly in $H^1(\Om)\times[H^1(\Om)]^{d\times d}$. The rest of the proof is about characterizing these limits. We proceed in three steps.

{\it Step 1: Trace constraint.}
To show that $\widetilde{s} = |s|$, we use a standard approximation estimate for the Lagrange interpolant and the fact that $|\nabla |s_h|| = |\nabla s_h|$ a.e.:
\[
\| \widetilde{s}_h - |s_h| \|_{L^2(\Om)} = \| I_h |s_h| - |s_h| \|_{L^2(\Om)} \le C h \| \nabla |s_h| \|_{L^2(\Om)} \le C \Lambda h.
\]
This, together with the triangle inequality and the fact that $s_h \to s$, $ \widetilde{s}_h \to  \widetilde{s}$ in $L^2(\Omega)$, give
\[ 
\big| \widetilde{s} - |s| \big| \le \big| \widetilde{s} - \widetilde{s}_h v|
+ \big| \widetilde{s}_h - |s_h| \big| + \big| |s_h| - |s| \big| \to 0
\quad \mbox{as } h \to 0.
\]
Using a similar argument, we can show that $s = \tr[\vU]$ and $\widetilde{s} = \tr [\widetilde{\vU}]$. Indeed, since $s_h = I_h(\tr[\vU_h])$, we have
\[
\| \tr[\vU_h] - s_h  \|_{L^2(\Om)} \le C h \| \nabla (\tr[\vU_h]) \|_{L^2(\Om)} \le C\Lambda h,
\]
and thus
\[
\big| \tr [\vU] - s \big| \le \big| \tr [\vU] - \tr [\vU_h] \big|
+ \big|\tr [\vU_h] - s_h \big| + \big|s_h - s\big|  \to 0 \quad \mbox{as } h \to 0.
\]

\medskip\noindent
{\it Step 2: Rank-one constraint.}
We now show that both $\vU$ and $\widetilde{\vU}$ have rank at most $1$; this is a new issue relative to \cite{Nochetto_SJNA2017}. In order to apply Lemma \ref{lem:rank1_characterization}, it suffices to check that 
\[
s^k = \tr[\vU^k], \quad \widetilde{s}^k = \tr[\widetilde{\vU}^k] \quad \forall k = 2,\ldots, d.
\]
Since the two identities above follow from the same argument, we just prove the first one. Let $2 \le k \le d$. The discrete admissibility condition \eqref{eqn:LdG_struct_condition_discrete} implies that $s_h^k(x_i) = \tr[\vU_h(x_i)^k]$ for all $x_i \in \Nk_h$, whence $I_h(s_h^k) = I_h(\tr[\vU_h^k])$. In a similar fashion as before, we use the triangle inequality to write
\[
\big| s^k - \tr[\vU^k] \big| \le \big| s^k - s_h^k \big| + \big| s_h^k - I_h(s_h^k) \big| + \big|I_h(\tr[\vU_h^k]) - \tr[\vU_h^k] \big| + \big| \tr[\vU_h^k] - \tr[\vU^k] \big|.
\]
The first and last terms in the right hand side tend to $0$ a.e., because $s_h \to s$ and $\vU_h \to \vU$. Next, we note that
$| \nabla s_h^k|  = k |s_h|^{k-1} | \nabla s_h | \le d | \nabla s_h |,$
because $|s_h| \le 1$, whence
\[
\| s_h^k - I_h (s_h^k) \|_{L^2(\Om)} \le C \Lambda h \to 0, \quad \mbox{as } h \to 0.
\]
The estimate
\[
\| I_h (\tr[\vU_h^k]) - \tr[\vU_h^k] \|_{L^2(\Om)} \le C \Lambda h \to 0, \quad \mbox{as } h \to 0,
\]
follows in a similar fashion. This proves that $\vU$ and $\widetilde{\vU}$ have rank $\leq 1$ a.e.

\medskip\noindent
{\it Step 3: Line field $\vNN$.}
Because $s = \tr[\vU]$, it follows that $\rank(\vU) = 1$ if and only if $s\ne 0$. Therefore, we can define a line field $\vNN : \Om\setminus\Sing \to \LL^{d-1}$ by $\vNN = s^{-1} \vU$, and extend $\vNN$ to $\Sing$ by any arbitrary tensor in $\LL^{d-1}$.

We next show that $\vNN_h \to \vNN$ a.e. in $\Om\setminus\Sing$ and in $L^2(\Om\setminus\Sing)$. We note that at every element $T \in \Tk_h$, the second derivatives of $s_h$ and $\vNN_h$ vanish, because these functions are piecewise linear. Thus, $\| s_h \vNN_h - I_h (s_h \vNN_h) \|_{L^1(T)} \le C h^2 \| \nabla s_h \otimes \nabla \vNN_h \|_{L^1(T)}$, and summing over all elements $T \in \Tk_h$, we obtain
\[
\| s_h \vNN_h - I_h (s_h \vNN_h) \|_{L^1(\Om)} \le C h^2 \| \nabla s_h \otimes \nabla \vNN_h \|_{L^1(\Om)} \le C h^2 \| \nabla s_h \|_{L^2(\Om)} \| \nabla \vNN_h \|_{L^2(\Om)}.
\]
Since $|\vNN_h| \le 1$, an inverse inequality yields $ \| \nabla \vNN_h \|_{L^2(\Om)} \le C h^{-1}$ and therefore
\begin{equation} \label{eqn:conv_s_hvNN_h}
\| s_h \vNN_h - I_h (s_h \vNN_h) \|_{L^1(\Om)} \le C \Lambda h \to 0 \quad \mbox{as } h \to 0.
\end{equation}
Noticing that $I_h (s_h \vNN_h) = \vU_h \to \vU$, we deduce that $s_h \vNN_h \to \vU$ a.e. in $\Om$ as $h \to 0$. Since $s_h \to s$ a.e., for almost every $x \in \Om \setminus \Sing$ it holds that $s_h(x) \ne 0$ if $h$ is sufficiently small, and we deduce
\[
\vNN_h (x) = \frac{s_h(x) \vNN_h (x)}{s_h(x)} \to \frac{\vU(x)}{s(x)} = \vNN(x) \quad \mbox{as } h \to 0.
\]
Convergence $\vNN_h \to \vNN$ in $L^2(\Om\setminus \Sing)$ now follows by the Dominated Convergence Theorem, as $|\vNN_h| \le 1$. Finally, to prove that $\widetilde \vU = \widetilde s \vNN$ a.e. in $\Om$, in the same fashion as \eqref{eqn:conv_s_hvNN_h} we can show that $\| \widetilde{s}_h \vNN_h - I_h (\widetilde{s}_h \vNN_h) \|_{L^1(\Om)} \to 0$ as $h \to 0$ which, recalling that $\widetilde{\vU}_h = I_h(\widetilde{s}_h \vNN_h) \to \widetilde{\vU}$, gives $\widetilde{s}_h \vNN_h \to \widetilde{\vU}$. Because $\widetilde{s}_h \to \widetilde{s}$ and $\vNN_h \to \vNN$ a.e. in $\Om \setminus \Sing$, it follows that $\widetilde{\vU} = \widetilde{s}\vNN$ a.e. in $\Om$.

\medskip
{\it Step 4: Lebesgue gradient and orthogonality.} At the Lebesgue points of $(\widetilde{s},\widetilde{\vU})$ and their weak gradients $(\nabla \widetilde{s}, \nabla\widetilde{\vU})$, the first order Taylor expansions exist and define superlinear approximations of $(\widetilde{s},\widetilde{\vU})$ in the $L^2$ sense \cite[Chapter 6.1.2]{EvansGariepy_book2015}. This defines $L^2$-gradients for $(\widetilde{s}.\widetilde{\vU})$ which coincide with the weak gradients. At each Lebesgue point $x\in\Omega\setminus\Sing$ of $(\widetilde{s},\vNN,\widetilde{\vU},\nabla \widetilde{s},\nabla\widetilde{\vU})$ we define the quantity $\nabla\vNN(x)$ to be
\[
\nabla \vNN(x) := \frac{\nabla \widetilde{\vU}(x) - \nabla \widetilde{s}(x) \otimes \vNN(x)}{\widetilde{s}(x)}.
\]
To verify that  $\nabla\vNN(x)$ is the $L^2$-gradient of $\vNN$ at $x$, we have to show that the first order Taylor expansion around $y=x$ gives a superlinear approximation of $\vNN(y)$ in the $L^2$ sense. Therefore, we let $B_\eps(x)$ denote the ball centered at $x$ of radius $\eps$ and observe that
\begin{align*}
  \fint_{B_\eps(x)} \Big| \vNN(y) - \vNN(x) - \nabla\vNN(x) (y-x) \Big|^2 dy
  & \lesssim \frac{1}{\widetilde{s}(x)^2} \fint_{B_\eps(x)} \Big| \widetilde{\vU}(y) - \widetilde{\vU}(x) - \nabla\widetilde{\vU}(x) (y-x)  \Big|^2 dy
  \\
  & + \frac{1}{\widetilde{s}(x)^2} \fint_{B_\eps(x)} \Big| \widetilde{s}(y) - \widetilde{s}(x) -\nabla \widetilde{s}(x) (y-x)\Big|^2 \big|\vNN(y)\big|^2 dy
  \\
  & + \frac{\big|\nabla \widetilde{s}(x)\big|^2}{\widetilde{s}(x)^2} \fint_{B_\eps(x)}  \Big|\vNN(y) - \vNN(x) \Big|^2 \,|y-x|^2 dy = o(\eps^2)
\end{align*}  
as $\eps\to0$ because the first order Taylor expansions of $(\widetilde{s},\widetilde{\vU})$ converge superlinearly at $x$, which is a Lebesgue point of $\vNN$ that belongs to $L^\infty(\Omega)$, and $\widetilde{s}(x)>0$ and $\nabla\widetilde{s}(x)$ are fixed.

We next claim that $\nabla\vNN : \nabla \widetilde{s} \otimes\vNN = 0$ and note that this is true if and only if $\nabla\widetilde{\vU} : \nabla \widetilde{s} \otimes \vNN=|\nabla \widetilde{s} \otimes \vNN|^2$ at any Lebesgue point $x \in \Om\setminus\Sing$ as above. To see this, we compute at $x$
\[
|\nabla \widetilde{s} \otimes \vNN|^2 = \sum_{i,j,k=1}^d (\partial_i \widetilde{s})^2 (\vNN_{j,k})^2 \sum_{i=1}^d (\partial_i \widetilde{s})^2 = |\nabla \widetilde{s}|^2,
\]
and
\begin{align*}
\nabla\widetilde{\vU} : \nabla \widetilde{s} \otimes \vNN &=
\sum_{i,j,k=1}^d \partial_i\widetilde{\vU}_{j,k} \, \partial_i \widetilde{s} \,\vNN_{j,k}
= \frac{1}{\widetilde{s}} \sum_{i=1}^d \partial_i \widetilde{s} \sum_{j,k=1}^d \partial_i \widetilde{\vU}_{j,k} \, \widetilde{s} \vNN_{j,k}
\\
& = \frac{1}{\widetilde{s}} \sum_{i=1}^d \partial_i \widetilde{s} \sum_{j,k=1}^d
\partial_i \widetilde{\vU}_{j,k} \, \widetilde{\vU}_{j,k} = \frac{1}{2\widetilde{s}} \sum_{i=1}^d \partial_i \widetilde{s} \, \partial_i |\widetilde{\vU}|^2 = \frac{1}{2\widetilde{s}} \sum_{i=1}^d \partial_i \widetilde{s} \, \partial_i \widetilde{s}^2
= \sum_{i=1}^d (\partial_i \widetilde{s})^2 = |\nabla \widetilde{s}|^2.
\end{align*}
This shows the orthogonality relation $|\nabla\widetilde{\vU}|^2=|\nabla \widetilde{s}|^2 + \widetilde{s}^2 |\nabla\vNN|^2$
at every Lebesgue point $x\in\Omega\setminus\Sing$ of $(\widetilde{s},\vNN,\widetilde{\vU},\nabla \widetilde{s},\nabla\widetilde{\vU})$, and concludes the proof.
\end{proof}

\subsection{$\Gamma$-convergence}
We have collected all the elements needed to prove the main theoretical result of this work. Using a standard argument \cite{Braides_book2002,Braides_book2014,DalMaso_book1993}, we can prove the convergence of discrete global minimizers.

\begin{theorem}[convergence of discrete global minimizers] \label{thm:LdG_convergence}
  Let $(s_h, \vNN_h, \vU_h) \in \Admisuni^h(g_h, \vR_h)$ be a sequence of global minimizers of the discrete total energy $\Eunit^h$ defined in \eqref{eqn:LdG_discrete_energy}. Then, every cluster point $(s,\vNN,\vU)$ belongs to $\Admisuni(g, \vR)$ and $(s,\vU)$ is a global minimizer of the continuous total energy $\EuniUtotal$ given in \eqref{eqn:total-energy-U}. Moreover, $\vNN$ admits a Lebesgue gradient a.e. in the set $\Omega\setminus\Sing$ so that the continuous main energy
\[
\Eunimain[s,\vNN] := \frac{d-1}{d} \int_{\Omega\setminus\Sing} |\nabla s|^2 + \frac12 \int_{\Omega\setminus\Sing} s^2 |\nabla \vNN|^2
\]    
is well defined and satisfies $\Eunimain[s,\vNN] = \EuniUmain[s,\vU]$.
\end{theorem}
\begin{proof}
If $\lim_{h \to 0} \Euni^h[s_h, \vNN_h] = \infty$, then $\Admisuni(g,\vR)$ is empty because otherwise Lemma \ref{lem:LdG_limsup} (lim-sup inequality) would imply the existence of a triple $(s_h,\vNN_h,\vU_h)\in\Admisuni^h(g_h,\vR_h)$ with uniformly bounded discrete total energy $\Euni^h[s_h, \vNN_h]$. In this case there is nothing to prove. We thus assume there is some $\Lambda >0$ such that
\[
\limsup_{h\to 0} \Eunit^h[s_h,\vNN_h] \le \Lambda.
\]
Applying Lemma \ref{lem:LdG_coercivity} (coercivity) and Lemma \ref{lem:LdG_limits} (characterization of limits), we can extract subsequences $(s_h, \vU_h) \to (s, \vU)$, $(\widetilde{s}_h, \widetilde{\vU}_h) \to (\widetilde{s}, \widetilde{\vU})$,  converging a.e. in $\Om$, strongly in $L^2(\Om)\times[L^2(\Om)]^{d\times d}$ and weakly in $H^1(\Om)\times[H^1(\Om)]^{d\times d}$, and such that the limits satisfy the structural condition \eqref{eqn:LdG_structural_conditions}. By Lemma \ref{lem:LdG_wlsc} (weak lower semi-continuity) and the energy inequality \eqref{eqn:energyequality_tilde}, we have
\[ \begin{split}
\EuniUmain[\widetilde{s}, \widetilde{\vU}] & = -\frac{1}{2d} \iO \big| \nabla |\tr [\widetilde{\vU}]| \big |^2 dx + \frac12 \iO \big| \nabla \widetilde{\vU} \big|^2 dx \\
& \le \liminf_{h \to 0} \left( -\frac{1}{2d} \iO | \nabla \tr [\widetilde{\vU}_h] |^2 dx + \frac12 \iO | \nabla \widetilde{\vU}_h |^2 dx \right) \le \liminf_{h \to 0} \Eunimain^h[s_h, \vNN_h].
\end{split}\]
Moreover, $\BulkLdG(s_h) \to \BulkLdG(s)$ a.e. in $\Om$ because $s_h \to s$ a.e., whence applying Fatou's Lemma yields
\[
\Ebulk[s] = \frac{1}{\Bulkcoef}  \iO \BulkLdG(s) dx \le \liminf_{h \to 0} \iO \frac{1}{\Bulkcoef}  \BulkLdG(s_h) dx = \liminf_{h \to 0} \Ebulk^h[s_h].
\]
We have thus shown that 
\begin{equation} \label{eqn:limit_Eunih}
\EuniUmain[\widetilde{s}, \widetilde{\vU}]  + \Ebulk[s] \le \liminf_{h \to 0} \Big(\Eunimain^h[s_h, \vNN_h] + \Ebulk^h[s_h]\Big) = \liminf_{h \to 0} \Eunit^h[s_h, \vNN_h].
\end{equation}

Next, we prove that $\EuniUmain[\widetilde{s}, \widetilde{\vU}] = \Eunimain[s, \vNN]$. This follows from the orthogonality relation $|\nabla \widetilde{\vU}|^2 = |\nabla \widetilde{s}|^2 + \widetilde{s}^2 |\nabla\vNN|^2$ of Lemma \ref{lem:LdG_limits} (characterization of limits), valid a.e. in $\Om\setminus\Sing$, as well as $|\nabla \widetilde{\vU}|=|\nabla \widetilde{s}|=0$ a.e. in $\Sing$ \cite[Ch. 5, Exercise 17]{Evans:book}. Therefore, making use of properties $\widetilde{s}=|s|$ (from Lemma \ref{lem:LdG_limits}) and $|\nabla \widetilde{s}| = |\nabla s|$, we infer that
\begin{equation*}
\EuniUmain[\widetilde{s},\widetilde{\vU}] =
-\frac{1}{2d} \int_{\Omega\setminus\Sing} |\nabla \widetilde{s}|^2 + \frac{1}{2} \int_{\Omega\setminus\Sing} |\nabla\widetilde{\vU}|^2
= \frac{d-1}{2 d} \int_{\Omega\setminus\Sing} |\nabla \widetilde{s}|^2 + \frac12 \int_{\Omega\setminus\Sing} \widetilde{s}^2 |\nabla \vNN|^2 = \Eunimain[s,\vNN]
\end{equation*}
This, together with \eqref{eqn:limit_Eunih}, shows that the total energy satisfies
\begin{equation} \label{eqn:limit_minimizer}
\Eunit[s,\vNN] \le \liminf_{h \to 0} \Eunit^h[s_h, \vNN_h] .
\end{equation}

Next, given $\eps >0$, we consider $(t, \vN, \vV) \in \Admisuni(g,\vR)$ such that
\[
\Eunit[t, \vN] \le \inf_{(\vt', \vN'), \in \Admisuni(g,\vR)}\Eunit[t', \vN'] + \eps/2
\]
and, in view of Proposition \ref{prop:LdG_regularization}, we can take $(t_\eps, \vN_\eps, \vV_\eps) \in \Admisuni(g,\vR)$ with $(t_\eps, \vV_\eps) \in W^{1,\infty}(\Om) \times [W^{1,\infty}(\Om)]^{d\times d}$ such that
\[ \begin{split}
\Eunimain[t_\eps, \vN_\eps] & = \EuniUmain[t_\eps, \vV_\eps] \le \EuniUmain[t, \vV] + \eps/4 =  \Eunimain[t, \vN] + \eps/4.
\end{split} \]
Moreover, because $t_\eps \to t$ a.e. in $\Om$ so does $\BulkLdG(t_\eps)\to\BulkLdG(t)$. Since \eqref{eqn:psi_nondegenerate} and \eqref{eqn:LdG_bound_s_eps} imply that $|\BulkLdG(t_\eps)|$ is uniformly bounded, we can apply the Dominated Convergence Theorem to deduce that
\[
\Ebulk[t] = \frac{1}{\Bulkcoef} \iO \lim_{\eps \to 0} \BulkLdG(t_\eps) \,  dx = \lim_{\eps \to 0}  \frac{1}{\Bulkcoef}  \iO \BulkLdG(t) \, dx = \lim_{\eps \to 0} \Ebulk[t_\eps]  .
\]
Therefore, we can find $(t_\eps,\vN_\eps,\vV_\eps) \in \Admisuni(g,\vR)$ such that
\begin{equation} \label{eqn:quasi_minimizer}
\Eunit[t_\eps, \vN_\eps] \le  \inf_{(\vt', \vN'), \in \Admisuni(g,\vR)}\Eunit[t', \vN'] + \eps.
\end{equation}

We next consider the  Lagrange interpolants $t_{\eps,h}=I_h (t_\eps),\vV_{\eps,h} = I_h(\vV_\eps)$, and set $\vN_{\eps,h}(x_i) = \vV_\eps(x_i)/t_\eps(x_i)$ if $t_\eps(x_i) \ne 0$ and $\vN_{\eps,h}(x_i)$ equal to any tensor in $\LL^{d-1}$ otherwise. By the same arguments as before, it follows that
\[
\Ebulk[t_\eps] = \frac{1}{\Bulkcoef}  \iO \lim_{h \to 0} \BulkLdG (t_{\eps,h}) dx = \lim_{h \to 0} \frac{1}{\Bulkcoef}  \iO \BulkLdG (t_{\eps,h}) dx = \lim_{h \to 0} \Ebulk^h[t_{\eps,h}] .
\]
Using Lemma \ref{lem:LdG_limsup} (lim-sup inequality) in conjunction with this estimate, we arrive at 
\[
\Eunit[t_\eps, \vN_\eps] = \lim_{h\to 0} \Eunit^h[t_{\eps,h}, \vN_{\eps,h}],
\]
and therefore, by \eqref{eqn:limit_minimizer} and \eqref{eqn:quasi_minimizer}, the total energies verify
\[
\Eunit[s,\vNN] \le \liminf_{h\to 0} \Eunit^h[s_h, \vNN_h] \le \lim_{h\to 0} \Eunit^h[t_{\eps,h}, \vN_{\eps,h}] \le
\inf_{(\vt', \vN'), \in \Admisuni(g,\vR)}\Eunit[t', \vN'] + \eps.
\]
Since $\eps > 0$ is arbitrary, this proves that $(s,\vNN)$ is a global minimizer of $\Eunit$.
\end{proof}

In case there is a unique global minimizer of the continuous total energy $\Eunit$, Theorem \ref{thm:LdG_convergence} implies that the entire sequence of discrete global energy minimizers converges to it strongly in $L^2$ and weakly in $H^1$. We also point out that a well-known result in $\Gamma$-convergence theory \cite{Kohn_PRSE1989} guarantees that, for every isolated local minimizer of $\Eunit$ there is a sequence of local minimizers of $\Eunit^h$ that converges to it in the same sense. However, in either case, because of the lack of continuous dependence on data as well as regularity theory, we cannot derive convergence rates.

\section{Computation of discrete minimizers} \label{sec:LdG_Grad_Flow}

We next discuss a gradient flow algorithm for the computation of discrete minimizers. Recall that, according to \eqref{eqn:LdG_discrete_energy}, we write the discrete total energy as 
\[
 \Eunit^h[s_h, \vNN_h] = \Eunimain^h [s_h, \vNN_h] + \Ebulk^h [s_h],
\]
with main and bulk energies
\begin{equation*}
\begin{split}
\Eunimain^h [s_h, \vNN_h] = \frac{d-1}{4d} \sum_{i, j = 1}^N k_{ij} \left( \dij s_h \right)^2 + \Euniring^h [s_h, \vNN_h], \quad \Ebulk^h [s_h] = \frac{1}{\Bulkcoef} \iO \BulkLdG (s_h) dx,
\end{split}
\end{equation*}
where $\Euniring^h [s_h, \vNN_h]$ is the interaction energy
\[
\Euniring^h [s_h, \vNN_h] = \frac{1}{4} \sum_{i, j = 1}^N k_{ij} \left(\frac{s_h(x_i)^2 + s_h(x_j)^2}{2}\right) |\dij \vNN_h|^2.
\]

\noindent
{\bf Tangential variations.}
The algorithm we discuss here is an alternating direction method that, at each step $k\ge0$, first performs a tangential variation on the current line field $\vNN_h=\vn_h^k\otimes\vn_h^k$, then normalizes the update, and finally performs a gradient flow step on the current degree of orientation $s_h$. The director field $\vn_h^k$ belongs to
\[
\Nh = \{ \vv_h \in [\Sh]^d : \vv_h (x_i) \in \Sp^{d-1} \ \forall x_i \in \Nk_h \},
\]
whereas a tangential variation $\vt_h^k$ belongs to the space
\[
\Vh (\vn_h^k) = \{ \vv_h \in [\Sh]^d : \vv_h (x_i) \cdot \vn_h^k(x_i) = 0 \ \forall x_i \in \Nk_h \}. 
\]
It is easy to see that tangential variations $\vT^k_h$ of $\vNN_h^k$ are of the form
\[
\vT^k_h = \vn^k_h \otimes \vt^k_h + \vt^k_h \otimes \vn^k_h
\]
with $\vt_h^k \in \Vh(\vn_h^k)$. However, in our algorithm we shall update the line field $\widehat{\vNN}_h^{k+1}$ by
\[
\widehat{\vNN}_h^{k+1} = \big( \vn_h^k + \vt_h^k \big) \otimes \big( \vn_h^k + \vt_h^k \big)
= \vNN_h^k + \vT_h^k + \vt_h^k \otimes \vt_h^k.
\]
The extra quadratic term can be handled if we have control of $\vt_h^k$ in an $H^1(\Om)$-type space. This dictates the metric of the gradient flow. Bartels and Raisch first proposed the metric $H^1(\Omega)$ provided $s_h^k>0$ is constant \cite{Bartels_bookch2014}. In our case, $s_h^k$ may vary across the domain and may even vanish to allow for the formation of defects. Near the singular set, where $s_h^k$ is small, it is critical to allow for relatively large variations $\vt_h^k$ in order to accelerate the algorithm. We achieve this via the weight $\omega=(s_h^k)^2$ and corresponding weighted $H^1$-norm
\begin{equation}\label{eqn:weighed-H1}
  \| \vv \|_{H^1_{\omega}(\Om)} := \left( \int_\Om |\vv(x)|^2 \, dx +
  \int_\Om |\nabla\vv(x)|^2 \, \omega(x) \, dx \right)^{1/2}.
\end{equation}
Moreover, $\vt_h^k$ must vanish on the Dirichlet part $\Gamma_\vNN=\Gamma_\vU$ of the boundary so that $\widehat\vNN_h^{k+1} = \vM$ on $\Gamma_\vNN$. We thus introduce the subspace $H^1_{\Gamma_\vNN}(\Om)$ of $H^1(\Omega)$ of functions with vanishing trace on $\Gamma_\vNN$. 

\medskip\noindent
{\bf Discrete gradient flow.} The algorithm reads as follows.
Given $(s_h^0, \vNN_h^0,\vU_h^0) \in \Admisuni^h(g_h, \vR_h)$, with $\vNN_h^0 = \vn_h^0 \otimes \vn_h^0$, and a time step $\tau > 0$, iterate Steps 1--3 for $k \ge 0$:

\begin{enumerate}[1.]
\item {\em Weighted tangent flow step for $\vNN_h$:} find $\vt^k_h \in \Vh(\vn_h^k) \cap [H^1_{\Gamma_\vNN}(\Om)]^d$ and $\vT^k_h = \vn^k_h \otimes \vt^k_h + \vt^k_h \otimes \vn^k_h$ such that
\begin{equation} \label{eqn:flow_theta}
\frac{1}{\tau}\int_\Omega \big( \vt^k_h \cdot \vv_h + \nabla \vt^k_h \colon \nabla \vv_h  |s^k_h|^2 \big) + \delta_{\vNN_h} \Euniring^h [s^k_h, \vNN^k_h + \vT^k_h; \vV_h] = 0
\end{equation}
for all $\vV_h = \vn^k_h \otimes \vv_h + \vv_h \otimes \vn^k_h, \ \vv_h \in  \Vh(\vn_h^k) \cap [H^1_{\Gamma_\vNN}(\Om)]^d$.
\smallskip

\item {\em Projection:} update $\vNN_h^{k+1} \in \THh$ by
\begin{equation} \label{eqn:vNN_update}
\vNN_h^{k+1}(x_i) := \frac{\vn_h^k(x_i) + \vt_h^k(x_i)}{|\vn_h^k(x_i) + \vt_h^k(x_i)|} \otimes \frac{\vn_h^k(x_i) +  \vt_h^k(x_i)}{|\vn_h^k(x_i) + \vt_h^k(x_i)|} \quad \forall x_i \in \Nk_h.
\end{equation}

\item {\em Gradient flow step for $s_h$:} find $s_h^{k+1} \in \Sh(g_h)$ such that
\begin{equation} \label{eqn:flow_s}
\frac{1}{\tau} \int_\Omega (s_h^{k+1} - s_h^k) \, z_h + \delta_{s_h} \Eunit^h[s_h^{k+1}, \vNN_h^{k+1} ; z_h ] = 0 \quad \forall z_h\in\Sh(0).
\end{equation}
\end{enumerate}
The symbols $\delta_{\vNN_h} \Eunimain^h$ and $\delta_{s_h} \Eunimain^h$ stand for the standard first variations of these functionals, whereas $\delta_{s_h} \Ebulk^h$ uses the following convex splitting method \cite{Shen_SJSC2010, Wise_SJNA2009} to obtain an unconditionally stable scheme. Let $\psi_c$, $\psi_e$ be convex functions so that the double-well potential splits as $\BulkLdG(s) = \psi_c(s) - \psi_e(s)$ and take
\begin{equation} \label{eqn:convex_concave}
\delta_{s_h} \Ebulk^h [s_h^{k+1}; z_h] := \frac{1}{\Bulkcoef} \iO \big( \psi_c' (s_h^{k+1}) - \psi_e' (s_h^{k}) \big) z_h \, dx \quad\forall \, z_h\in\Sh(0).
\end{equation}

\medskip\noindent
{\bf Energy decrease property.}
Note that the discrete interaction energy \eqref{eqn:LdG_def_Euniringh} can be written equivalently as
\[
\Euniring^h [s_h^k , \vNN_h^k] = \frac18 \sum_{i,j} k_{ij} \Big(s_h(x_1)^2+s_h(x_j)^2 \Big) \Big(1 - \vNN_h^k(x_i) \colon \vNN_h^k(x_j) \Big).
\]
To show that Step 2 decreases this energy, namely
\begin{equation} \label{eqn:decreasing_energy}
\Euniring^h[s_h^k, \vNN_h^{k+1}] \le \Euniring^h[s_h^k, \vttNN_h^{k+1}],
\end{equation}
we recall that $k_{ij}\ge0$ if $i\ne j$ and invoke the following result from \cite[Lemmas 3 and 4]{Bartels_bookch2014}, but omit its proof.

\begin{lemma}[monotonicity] \label{lem:normalization}
Let the mesh $\Tk_h$ be weakly acute (cf. \eqref{eqn:weakly-acute}) and let $\vv_h \in \Uh$ be such that $|\vv_h(x_i)| \ge 1$ for all $x_i \in \Nk_h$. The discrete tensor fields $\vV_h,\widetilde{\vV}_h \in \Uh$,
\[
\vV_h(x_i) = \vv_h(x_i) \otimes \vv_h(x_i), \quad \widetilde{\vV}_h(x_i) = \frac{\vv_h(x_i)}{|\vv_h(x_i)|} \otimes \frac{\vv_h(x_i)}{|\vv_h(x_i)|}.
\]
satisfy the inequality
\[
1 -  \vV_h(x_i) \colon \vV_h(x_j) \le \frac12 | \dij \widetilde{\vV}_h |^2.
\]
\end{lemma}

We also need the following key property of \eqref{eqn:convex_concave} (cf. \cite[Lemma 4.1]{Nochetto_SJNA2017}, for example).

\begin{lemma}[convex-concave splitting] \label{lem:convex_concave}
Given $s_h^k, s_h^{k+1} \in \Sh$, we have
\[
\iO \BulkLdG(s_h^{k+1}) \, dx - \iO \BulkLdG(s_h^k) \, dx \le \delta_{s_h} \Ebulk^h [s_h^{k+1}; s_h^{k+1}-s_h^k].
\]
\end{lemma}

Next, we prove that the discrete gradient flow algorithm is energy-decreasing.

\begin{theorem}[energy decrease] \label{thm:energy_decrease}
If the meshes are weakly acute and $\tau\le C_0 h^{d/2}$, with $C_0$ proportional to $\Eunit^h[s_h^0, \vNN_h^0]^{-1/2}$, then it holds that 
\[
\Eunit^h[s_h^K, \vNN_h^K] + \frac{1}{2\tau} \left( \sum_{k=0}^{K-1} \| \vt_h^k \|^2_{H^1_\omega(\Omega)}  + \| s_h^{k+1} - s_h^k \|_{L^2(\Om)}^2 \right) \le \Eunit^h[s_h^0, \vNN_h^0] \quad \forall K \ge 1,
\]
where $H^1_\omega(\Omega)$ is the weighted Sobolev space defined in \eqref{eqn:weighed-H1}. 
Therefore, the algorithm stops in a finite number of steps: given a tolerance $\varepsilon$, there exists $K=K_\eps \ge 1$ such that $\frac{1}{\tau}  (\| \vt_h^K \|^2_{H^1_\omega(\Omega)}  + \| s_h^{K} - s_h^{K-1} \|^2 )< \varepsilon .$
\end{theorem}
\begin{proof}
We proceed as in \cite[Lemma 6]{Bartels_bookch2014} except for the presence of the variable order parameter $s_h^k$ and the weighted $H^1_\omega(\Omega)$ metric. We make the induction assumption that
\[
\Eunit^h [s_h^k, \vNN_h^k] \le \Lambda := \Eunit^h [s_h^0, \vNN_h^0].
\]
for $k\ge0$ and show the estimate
\[
\frac{1}{2\tau} \left( \| \vt_h^k \|^2_{H^1_\omega(\Omega)}  + \| s_h^{k+1} - s_h^k \|_{L^2(\Om)}^2 \right) + \Eunit^h[s_h^{k+1}, \vNN_h^{k+1}] \le \Eunit^h[s_h^{k}, \vNN_h^k].
\]
Upon summation on $k$ this implies the asserted estimate.
We split the proof into several steps.

\medskip\noindent
{\em Step 1: Explicit expression} for the solution to \eqref{eqn:flow_theta}.
In order to simplify the notation, we write
\begin{equation} \label{eqn:def_sigma_ij} \begin{aligned}
& \sigma_{ij} := k_{ij} \, \frac{s_h^{k}(x_i)^2 + s_{h}^{k}(x_j)^2}{2} \ge 0, \mbox{ if } i \neq j,\\
& \vtNN_h^{k+1} := \vNN_h^k + \vT_h^k.
\end{aligned}\end{equation}
We set $\vv_h = \vt_h^k$ in \eqref{eqn:flow_theta}, and thus $\vV_h = \vT_h^k$, to obtain
\[
\frac{1}{\tau} \| \vt_h^k \|^2_{H^1_\omega(\Omega)} + \frac12 \sum_{i,j} \sigma_{ij} (\dij\vtNN_h^{k+1}) \colon (\dij\vT_h^k) = 0.
\] 
The elementary equality
$2(\delta_{ij} \vtNN_h^{k+1}) : (\delta_{ij} \vT_h^k) = | \delta_{ij} \vtNN_h^{k+1} |^2
- | \delta_{ij} \vtNN_h^{k} |^2 + |\delta_{ij} \vT_h^k|^2$ and \eqref{eqn:LdG_def_Euniringh}
yield
\[ \begin{aligned}
\frac12  \sum_{i,j} \sigma_{ij} (\dij\vtNN_h^{k+1}) \colon (\dij\vT_h^k)
& = \Euniring^h[s_h^k, \vtNN_h^{k+1}] - \Euniring^h[s_h^k, \vNN_h^{k}] + \Euniring^h[s_h^k, \vT_h^{k}],
\end{aligned} \]
and therefore we deduce
\begin{equation} \label{eqn:bound_theta}
\frac{1}{\tau} \| \vt_h^k \|^2_{H^1_\omega(\Omega)} + \Euniring^h[s_h^k, \vtNN_h^{k+1}] + \Euniring^h[s_h^k, \vT_h^{k}] =  \Euniring^h[s_h^k, \vNN_h^{k}] .
\end{equation}

\medskip\noindent
{\em Step 2: Monotonicity of projection.}
We define the updated line field to be
\[
\vttNN_h^{k+1} = (\vn_h^k + \vt_h^k) \otimes (\vn_h^k + \vt_h^k),
\]
and recall that $\vNN^{k+1}$ defined in \eqref{eqn:vNN_update} is its nodewise normalization. From Lemma \ref{lem:normalization}, we have the monotonocity relation \eqref{eqn:decreasing_energy}:
\begin{equation*} 
\Euniring^h[s_h^k, \vNN_h^{k+1}] \le \Euniring^h[s_h^k, \vttNN_h^{k+1}].
\end{equation*}

\medskip\noindent
{\em Step 3: Bound of the energy $\Euniring^h[s_h^k, \vttNN_h^{k+1}]$.}
Expanding the expression for $\vttNN_h^{k+1}$, we have 
\[
\vttNN_h^{k+1} = \vtNN_h^{k+1} + \vt_h^k \otimes \vt_h^k .
\]
Therefore, by Cauchy-Schwarz,
\[ \begin{aligned}
\Euniring^h[s_h^k, \vttNN_h^{k+1}] & = \Euniring^h[s_h^k, \vtNN_h^{k+1}] + \Euniring^h[s_h^k, \vt_h^k \otimes \vt_h^k] + \frac12 \sum_{i,j} \sigma_{ij} (\dij \vtNN_h^{k+1}) \colon \dij (\vt_h^k \otimes \vt_h^k) \\
& \le  \Euniring^h[s_h^k, \vtNN_h^{k+1}] + \Euniring^h[s_h^k, \vt_h^k \otimes \vt_h^k] + 2 \Euniring^h[s_h^k, \vtNN_h^{k+1}]^{1/2}  \Euniring^h[s_h^k, \vt_h^k \otimes \vt_h^k]^{1/2},
\end{aligned} \]
whence
\begin{equation} \label{eqn:bound_thetatilde}
\begin{aligned}  
  \Euniring^h[s_h^k, \vttNN_h^{k+1}] &\le \Euniring^h[s_h^k, \vtNN_h^{k+1}]
  \\ &+  \Euniring^h[s_h^k, \vt_h^k \otimes \vt_h^k]^{1/2} \left( \Euniring^h[s_h^k, \vt_h^k \otimes \vt_h^k]^{1/2} + 2\Euniring^h[s_h^k, \vtNN_h^{k+1}]^{1/2} \right).
\end{aligned}
\end{equation}
Invoking the induction hypothesis, we readily see that $\Euniring^h [s_h^k, \vNN_h^k] \le \Lambda$, and using \eqref{eqn:bound_theta} gives
\begin{equation}\label{eqn:induction}
\frac{1}{\tau}\| \vt_h^k \|^2_{H^1_\omega(\Omega)} +
  \Euniring^h [s_h^k, \vtNN_h^{k+1}]  \le \Lambda.
\end{equation}
To bound $\Euniring^h[s_h^k, \vt_h^k \otimes \vt_h^k]$, we write
\[ \begin{aligned}
\dij(\vt_h^k \otimes \vt_h^k) 
& = \dij \vt_h^k \otimes \vt_h^k(x_j) + \vt_h^k(x_i) \otimes  \dij \vt_h^k,
\end{aligned} \]
and thereby obtain
\[
\dij (\vt_h^k \otimes \vt_h^k) \colon \dij (\vt_h^k \otimes \vt_h^k) \le C | \dij \vt_h^k |^2 \max \big\{ |\vt_h^k(x_i)|, |\vt_h^k(x_j)| \big\}^2 .
\]
Using \eqref{eqn:def_sigma_ij}, we deduce
\begin{align*} \label{eqn:first-bound-nonlinear-tk}
\Euniring^h[s_h^k, \vt_h^k \otimes \vt_h^k]  &\le C \sum_{i,j} \sigma_{ij}  | \dij \vt_h^k |^2  \max \big\{ |\vt_h^k(x_i)|, |\vt_h^k(x_j)| \big\}^2
\\
& \le C \sum_{T \in \mathcal{T}_h} | \vt_h^k |^2_{H^1_\omega(T)} \left\| \vt_h^k \right\|_{L^\infty(T)}^2 \le C | \vt_h^k |^2_{H^1_\omega(\Omega)} \| \vt_h^k \|_{L^\infty(\Omega)}^2.
\end{align*}
Since the mesh $\Tk_h$ is shape regular and quasi-uniform, we resort to the inverse inequality $\| \vt_h^k \|_{L^\infty(\Omega)} \le C h^{-d/2} \| \vt_h^k \|_{L^2(\Omega)}$ and rewrite the above expression as follows:
\[
\Euniring^h[s_h^k, \vt_h^k \otimes \vt_h^k] \le C h^{-d} | \vt_h^k |^2_{H^1_\omega(\Omega)} \| \vt_h^k \|_{L^2(\Omega)}^2 \le C h^{-d} \| \vt_h^k \|^4_{H^1_\omega(\Omega)}
\]
Consequently, \eqref{eqn:induction} yields the bound
\[
\Euniring^h[s_h^k, \vt_h^k \otimes \vt_h^k]^{1/2} + 2\Euniring^h[s_h^k, \vtNN_h^{k+1}]^{1/2}
\le C h^{-d/2} \tau \Lambda + 2 \Lambda^{1/2} \le 4 \Lambda^{1/2},
\]
provided $\tau \le C \Lambda^{-1/2} h^{d/2}$. Inserting this expression into \eqref{eqn:bound_thetatilde} results in
\[
\Euniring^h[s_h^k, \vttNN_h^{k+1}] \le \Euniring^h[s_h^k, \vtNN_h^{k+1}]
+ C  h^{-d/2} \Lambda^{1/2} \|\vt_h^k\|_{H^1_\omega(\Omega)}^2.
\]
 
\medskip\noindent
{\em Step 4: Bound of the energy $\Euni^h[s_h^k,\vNN_h^{k+1}]$.}
Combining this estimate with \eqref{eqn:bound_theta} and \eqref{eqn:decreasing_energy}, we find
\begin{align*}
\Euniring^h[s_h^k,\vNN_h^k] & \ge \frac{1}{\tau} \| \vt_h^k \|^2_{H^1_\omega(\Omega)} +
\Euniring^h[s_h^k, \vtNN_h^{k+1}]
\\
& \ge \frac{1}{\tau} \Big(1 - C \Lambda^{1/2} h^{-d/2} \tau  \Big) \| \vt_h^k \|^2_{H^1_\omega(\Omega)} + \Euniring^h[s_h^k, \vttNN_h^{k+1}]
\\
& \ge \frac{1}{2\tau} \| \vt_h^k \|^2_{H^1_\omega(\Omega)} + \Euniring^h[s_h^k, \vNN_h^{k+1}],
\end{align*}
provided $\tau \le C \Lambda^{-1/2} h^{d/2}$ with a geometric constant $C$ perhaps smaller than before. Since the scalar variable $s_h^k$ remains fixed in the gradient flow for $\vNN_h^{k+1}$, adding $\Eunis^h[s_h^k]$ to both sides of the above inequality gives
\begin{equation} \label{eqn:bound_energy_k}
\frac{1}{2\tau}  \| \vt_h^k \|^2_{H^1_\omega(\Omega)} + \Eunimain^h[s_h^k, \vNN_h^{k+1}] \le \Eunimain^h[s_h^k,\vNN_h^k].
\end{equation}

\medskip\noindent
{\em Step 5: Gradient flow for $s_h$.}
Taking $z_h=s_h^{k+1}-s_h^k\in\Sh(0)$ in step 3 of the algorithm, and using the
  elementary identity
\[
2 s_h^{k+1} \big(s_h^{k+1} - s_h^k  \big) = \big| s_h^{k+1} \big|^2 - \big| s_h^k \big|^2
+ \big| s_h^{k+1} - s_h^k  \big|^2
\]
we readily obtain
\[
\Eunimain^h[s_h^{k+1},\vNN_h^{k+1}] - \Eunimain^h[s_h^k,\vNN_h^{k+1}]
\le \delta_{s_h} \Eunimain^h[s_h^{k+1},\vNN_h^{k+1};s_h^{k+1} - s_h^k.]
\]
In addition, applying Lemma \ref{lem:convex_concave} leads to
\[
\Ebulk^h[s_h^{k+1}] - \Ebulk^h[s_h^k] \le \delta_{s_h} \Ebulk^h[s_h^{k+1};s_h^{k+1} - s_h^k],
\]
and together with the previous inequality implies
\[
\Eunit^h[s_h^{k+1},\vNN_h^{k+1}] - \Eunit^h[s_h^k,\vNN_h^{k+1}] \le
\delta_s \Eunit^h[s_h^{k+1},\vNN_h^{k+1};s_h^{k+1}-s_h^k] = -\frac{1}{\tau}
\| s_h^{k+1} - s_h^k \|_{L^2(\Omega)}^2.
\]
Adding this expression to \eqref{eqn:bound_energy_k} yields the desired estimate and
completes the proof.
\end{proof}

\vspace{-0.1in}

\begin{remark}[CFL condition]
The stability constraint $\tau\le C \Eunit[s_h^0,\vNN_h^0]^{-1/2} h^{d/2}$ is due
to the weighted $H^1_\omega(\Omega)$ norm and the use of an inverse estimate between
$L^\infty(\Omega)$ and $L^2(\Omega)$. If the weight $\omega = (s_h^k)^2$ is bounded
away from zero, then the CFL condition is milder, namely
$\tau\le C \Eunit[s_h^0,\vNN_h^0]^{-1/2} h^{d/2-1}|\log h|$ \cite{Bartels_bookch2014}.
The weight $\omega$ is critical because it accelerates the algoritm upon allowing large
variations of $\Theta_h^k$ near defects where it becomes small. 
\end{remark}

\vspace{-0.1in}
\section{Numerical Experiments} \label{sec:LdG_experiments}

To illustrate our method, we present computational experiments carried out with the MATLAB/C++ toolbox FELICITY \cite{Walker_SJSC2018}. We first consider a problem for the Landau - de Gennes energy with orientable Dirichlet boundary conditions. In such a case, the resulting line field of degree $+1$ is orientable, and the energy minimization problem is equivalent to the one given by minimizing the Ericksen energy; this allows us to compare with \cite{Nochetto_SJNA2017}. Afterwards, we illustrate the method's ability to capture non-orientable defects of degree $+1/2$ in two and three dimensional experiments, the latter leading to a non-straight line defect. We conclude with a Saturn-ring defect of degree $-1/2$ around a colloidal spherical inclusion.

\subsection{Ericksen vs. Landau de Gennes} \label{sec:orientable_experiment}

It is known that, if the line field is orientable, then a director field representation is equivalent. Thus, we compare the solutions for the Ericksen and the Landau - de Gennes model with orientable boundary conditions. 
In this first experiment we are {\em not taking into account the double-well potential.} If $\vNN = \vm \otimes \vm$ is an orientable line field, then a straightforward calculation gives $|\nabla \vNN|^2 = 2 | \nabla \vm|^2$, and therefore 
\[
\Eunimain [s, \vNN] = \frac{d-1}{2d} \iO |\nabla s|^2 dx + \iO s^2 |\nabla \vm|^2 dx  = 2 \Eerkmain [s, \vm],
\]
where the Ericksen energy corresponds to $\kappa = \frac{d-1}{2d}$.

We consider $\Omega = (0,1)^2$, and impose the Dirichlet boundary conditions on $\dOm$: 
\[
s = \frac12, \quad \vn = \frac{(x,y) - (1/2, 1/2)}{|(x,y) - (1/2, 1/2)|}, \quad \vNN = \vn \otimes \vn,
\]
and compare the minimizers of the discrete energies $\Eerkmain^h$ (with $\kappa = \frac14$) and $\Eunimain^h$. We initialize both gradient flows with $s = 1/2$ and a point defect away from the center. Figure \ref{fig:orientable} shows the equilibrium configurations for both models. For the solutions displayed, we computed $\Eunimain^h[s_{h,LdG},\vNN] = \Eerkmain^h[s_{h,Erk},\vn] \approx  1.234$, although $\min(s_{h,LdG}) \approx 2.3 \times 10^{-4}$ while $\min(s_{h,Erk}) \approx 5.8 \times 10^{-5}$.

\begin{center}
 \begin{figure}[ht]
\includegraphics[width=0.45\linewidth]{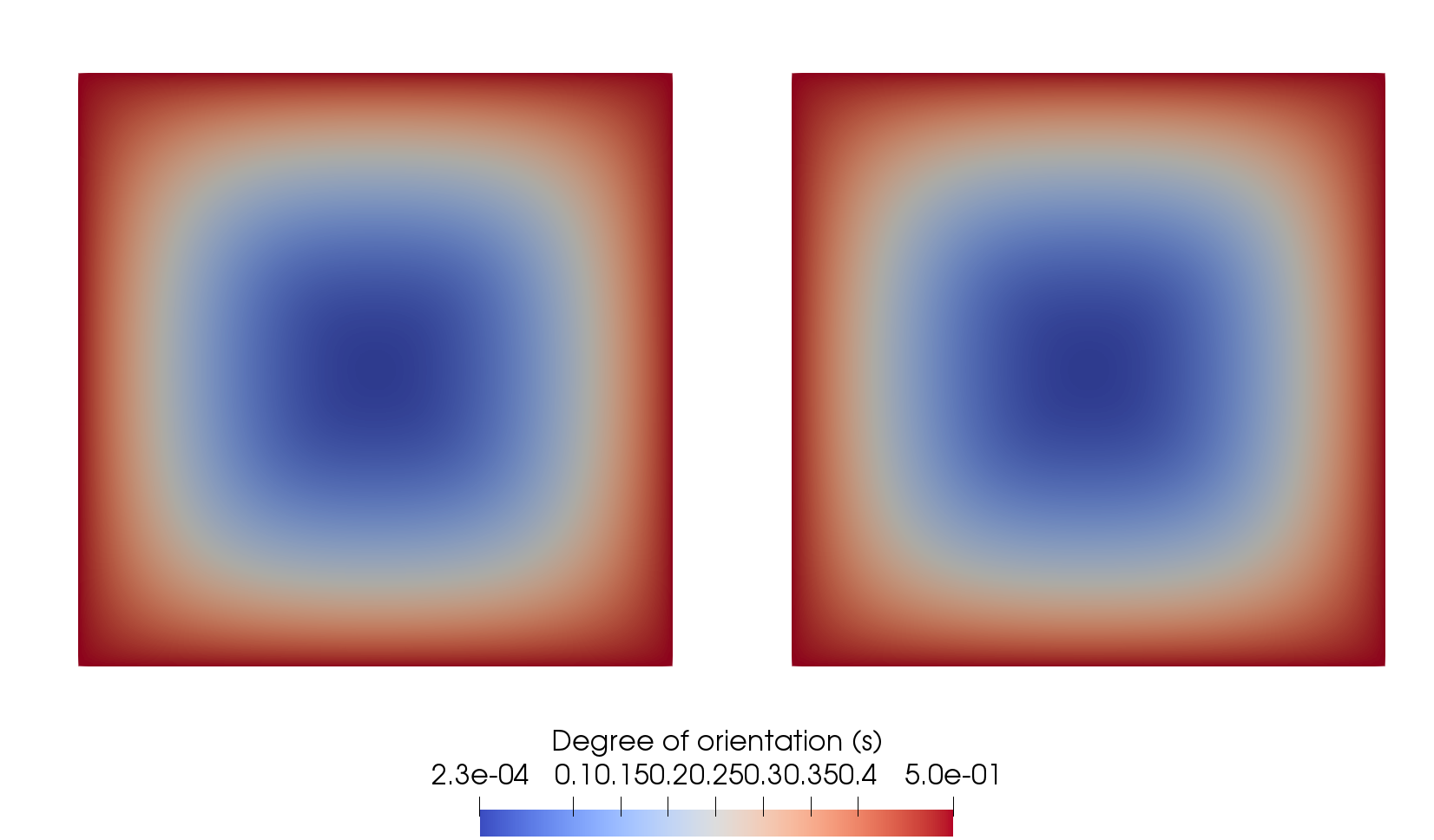} \hspace{0.2cm}
\includegraphics[width=0.45\linewidth]{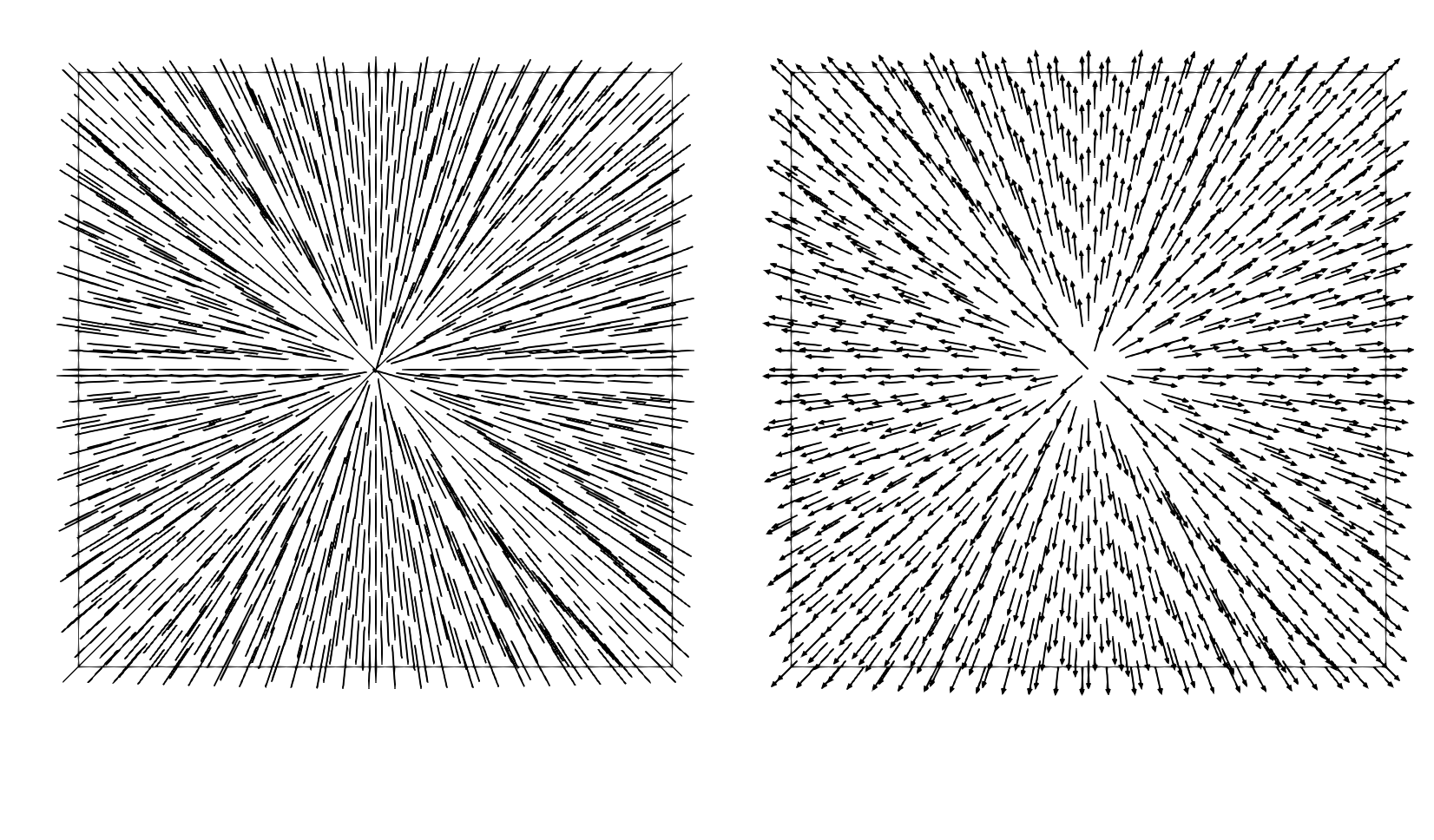}\\
\caption{Minimizing configurations for the Landau-de Gennes and Ericksen energies in 2-D for the setting discussed in Section \ref{sec:orientable_experiment}. Left: degree of orientation for both models (left is uniaxial Landau-de Gennes, right is Ericksen).  Right: line field $\vNN$ (left) and director field $\vn$ (right) are displayed.  In this case, the line field is orientable, so both the Ericksen model and uniaxially constrained model give the same result.
}
\label{fig:orientable}
\end{figure}
\end{center}

\subsection{Non-orientable field in two dimensions} \label{sec:half_plane}

Next, we simulate a non-orientable defect in the unit square $\Om = (0,1)^2$. We set the double-well potential with a convex splitting
\begin{equation*}
\begin{split}
\BulkLdG (s) & = \psi_c(s) - \psi_e(s) \\ 
&  := (26.20577 s^2 + 1) - (-4.1649313 s^4 + 30.2874 s^2),
\end{split}
\end{equation*}
with $\Bulkcoef = 1/16$, and note that $\BulkLdG$ has a local maximum at $s = 0$ and a global minimum at $s = s^* := 0.7$ with $\BulkLdG(s^*) = 0$ (by symmetry in two dimensions, $\BulkLdG(-s^*) = 0$). 
We impose Dirichlet boundary conditions for both $s$ and $\vNN$ on $\Gamma_s = \Gamma_\vNN = \dOm$,
\begin{equation}\label{eqn:plus_1/2_degree_defect_centered}
s = s^*, \quad \vn(x,y) = (\cos \theta, \sin \theta), \quad \vNN = \vn \otimes \vn, 
\quad \theta(x,y) = \frac12 \mbox{atan2} \left(\frac{y-1/2}{x-1/2}\right),
\end{equation}
where $\mbox{atan2}$ is the four-quadrant inverse tangent function, i.e. the boundary conditions for $\vNN$ correspond to a $+1/2$ degree defect centered at $(0.5,0.5)$. 
We initialize the gradient flow with $s = s^*$ and $\vNN$ corresponding to a $+1/2$ degree defect located at $(0.7167,0.2912)$, which has initial energy $\Euni^h[s_h,\vNN_h] = 18.5468$. We show the final equilibrium configurations of $s$ and the tensor field $\vNN$ in Figure \ref{fig:half_plane}. The method clearly captures the non-orientable defect at the domain center.  The final state has $\Euni^h[s_h,\vNN_h] = 2.1192$ and $\min(s_{h}) \approx 4.734 \times 10^{-3}$. 

\begin{center}
\begin{figure}[ht]
\includegraphics[width=0.4\linewidth]{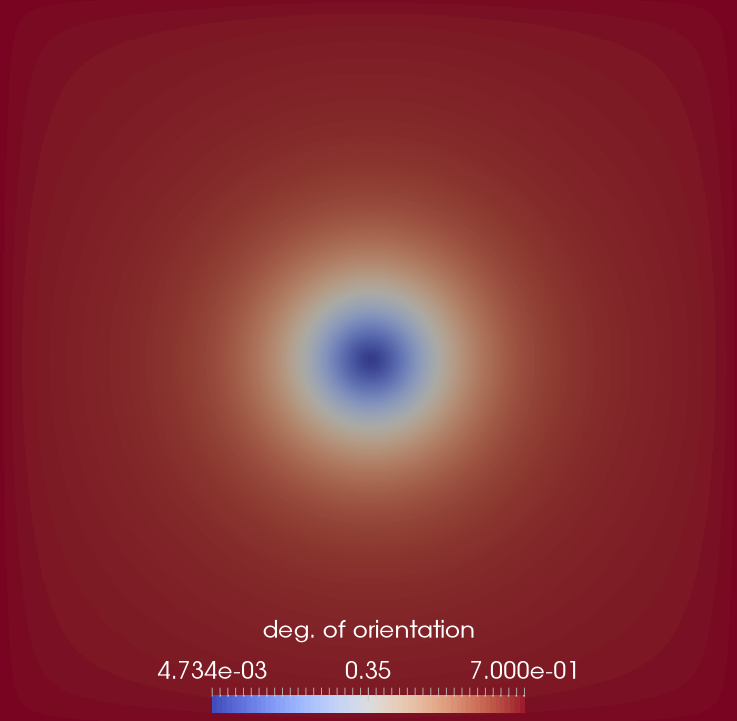} \hspace{0.3cm}
\includegraphics[width=0.4\linewidth]{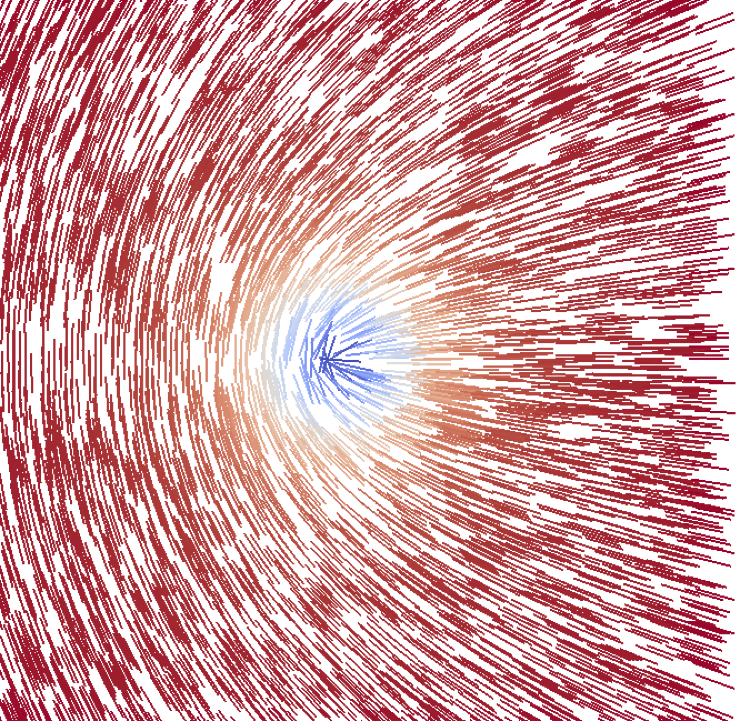}\\
\caption{A $+1/2$ degree point defect in 2-D (Section \ref{sec:half_plane}). Left: the degree-of-orientation $s$ is plotted with the singular region at the center. Right: the line field $\vNN$ is plotted and colored based on $s$.  The time step for the gradient flow was $\tau = 10^{-2}$. This configuration cannot be captured by the Ericksen (director) model.}
\label{fig:half_plane}
\end{figure}
\end{center}

\subsection{Line defect in three dimensions}\label{sec:line_defect}

We simulate a non-orientable line defect in the unit cube $(0,1)^3$.  The double-well potential with a convex splitting is given by
\begin{equation*}
\begin{split}
\BulkLdG (s) & = \psi_c(s) - \psi_e(s) \\ 
&  := (36.7709 s^2 + 1) - (-7.39101 s^4 + 4.51673 s^3 + 39.27161 s^2),
\end{split}
\end{equation*}
with $\Bulkcoef = 1/16$, and note that $\BulkLdG$ has a local maximum at $s = 0$ and a global minimum at $s = s^* := 0.700005531$ with $\BulkLdG(s^*) = 0$.

The boundary conditions for $\vNN$ were constructed in the following way.  Let $\theta_{0}(x,y)$ define a $+1/2$ degree defect in the plane, located at $(0.3,0.3)$ similar to \eqref{eqn:plus_1/2_degree_defect_centered}.  Likewise, let $\theta_{1}(x,y)$ define a $+1/2$ degree defect in the plane, located at $(0.7,0.7)$.  Next, define the Dirichlet boundary $\Gamma_s = \Gamma_\vNN = \dOm \setminus \Gamma_{o}$, where $\Gamma_{o} := \overline{\Om} \cap (\{ z=0 \} \cup \{ z=1 \})$.  Then, the Dirichlet conditions are
\begin{equation*}
s = s^*, \quad \vn(x,y) = (\cos \theta, \sin \theta,0), \quad \vNN = \vn \otimes \vn, 
\quad \theta(x,y,z) = (1 - z) \theta_{0} (x,y) + z \theta_{1} (x,y) + \pi z,
\end{equation*}
with vanishing Neumann condition on $\Gamma_{o}$.  Basically, the boundary conditions consist of rotating a planar $+1/2$ degree point defect as a function of $z$.  The solution is computed with the gradient flow approach \eqref{eqn:flow_theta} and time step $\tau = 10^{-3}$, and initialized with
\[
s = s^*, \quad \vn = (\cos \alpha, \sin \alpha, 0), \quad \vNN = \vn \otimes \vn, \quad \alpha(x,y,z) = \theta_{2} (x,y) + \pi z,
\]
where $\theta_{2} (x,y)$ corresponds to a $+1/2$ degree defect centered at $(0.5,0.5)$; this configuration has an initial energy of $\Euni^h[s_h,\vNN_h] = 10.013214$.

Figure \ref{fig:half_line_defect_view_A} shows three dimensional views of the minimizing configuration, where as Figure \ref{fig:half_line_defect_view_B} shows four horizontal slices of the solution.  A non-orientable line defect is observed, with final energy $\Euni^h[s_h,\vNN_h] = 5.2042593769$ and $\min(s_{h}) \approx 2.145 \times 10^{-2}$.
\begin{center}
\begin{figure}[ht]
\includegraphics[width=0.35\linewidth]{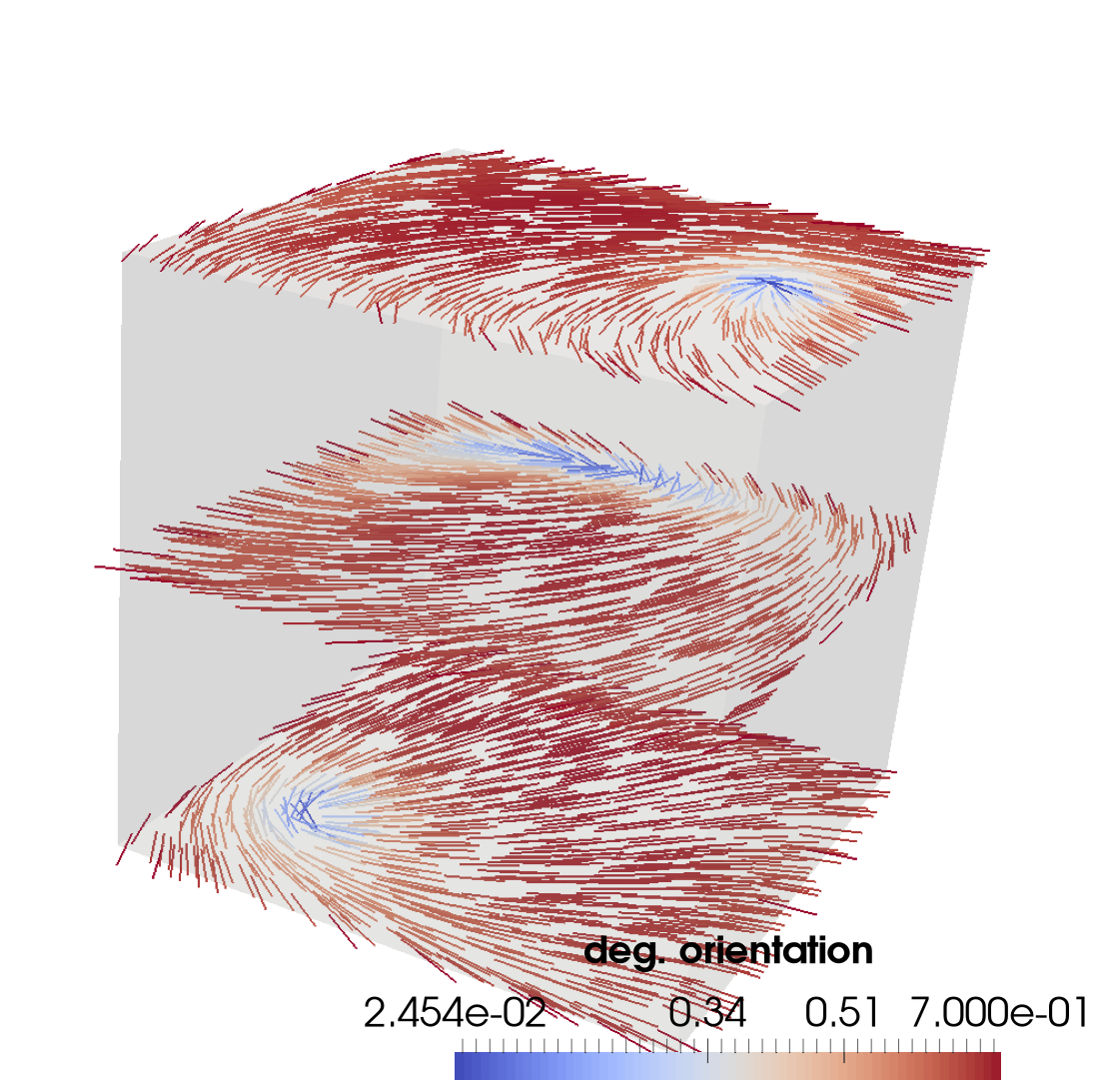} \hspace{0.2in}
\includegraphics[width=0.35\linewidth]{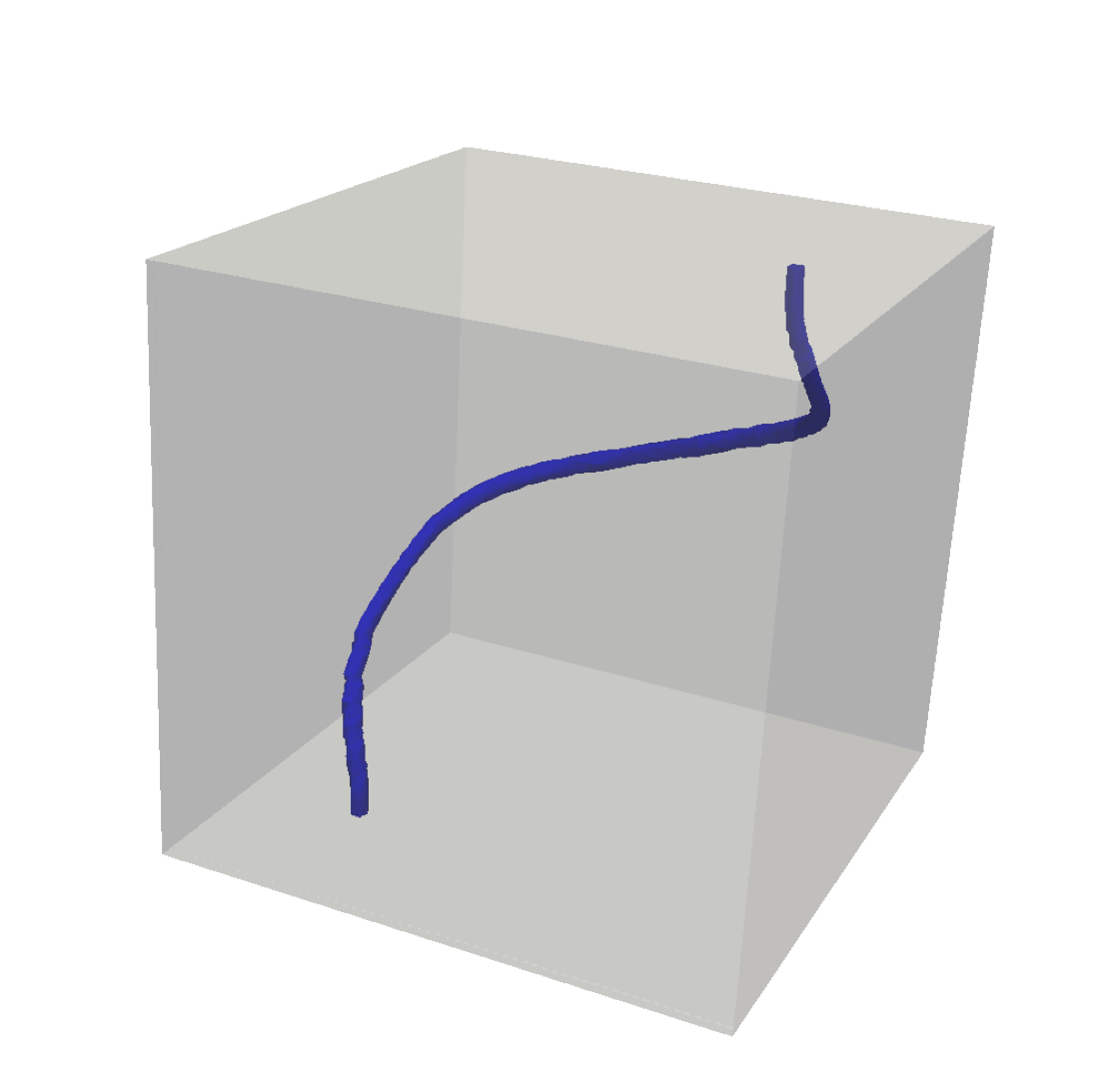}
\caption{A $+1/2$ degree line defect in a 3-D cube domain (Section \ref{sec:line_defect}). Left: line field $\vNN$ is shown at levels $z=0.0$, $0.5$, $1.0$ (colored by $s$).  Right: The $s=0.05$ iso-surface is shown that contains the line defect.  In each horizontal plane, the line field exhibits a $+1/2$ degree \emph{point} defect in 2-D.  The twisting of the line defect is due to the choice of boundary conditions.
}
\label{fig:half_line_defect_view_A}
\end{figure}
\end{center}

\begin{center}
\begin{figure}[ht]
\includegraphics[width=0.4\linewidth]{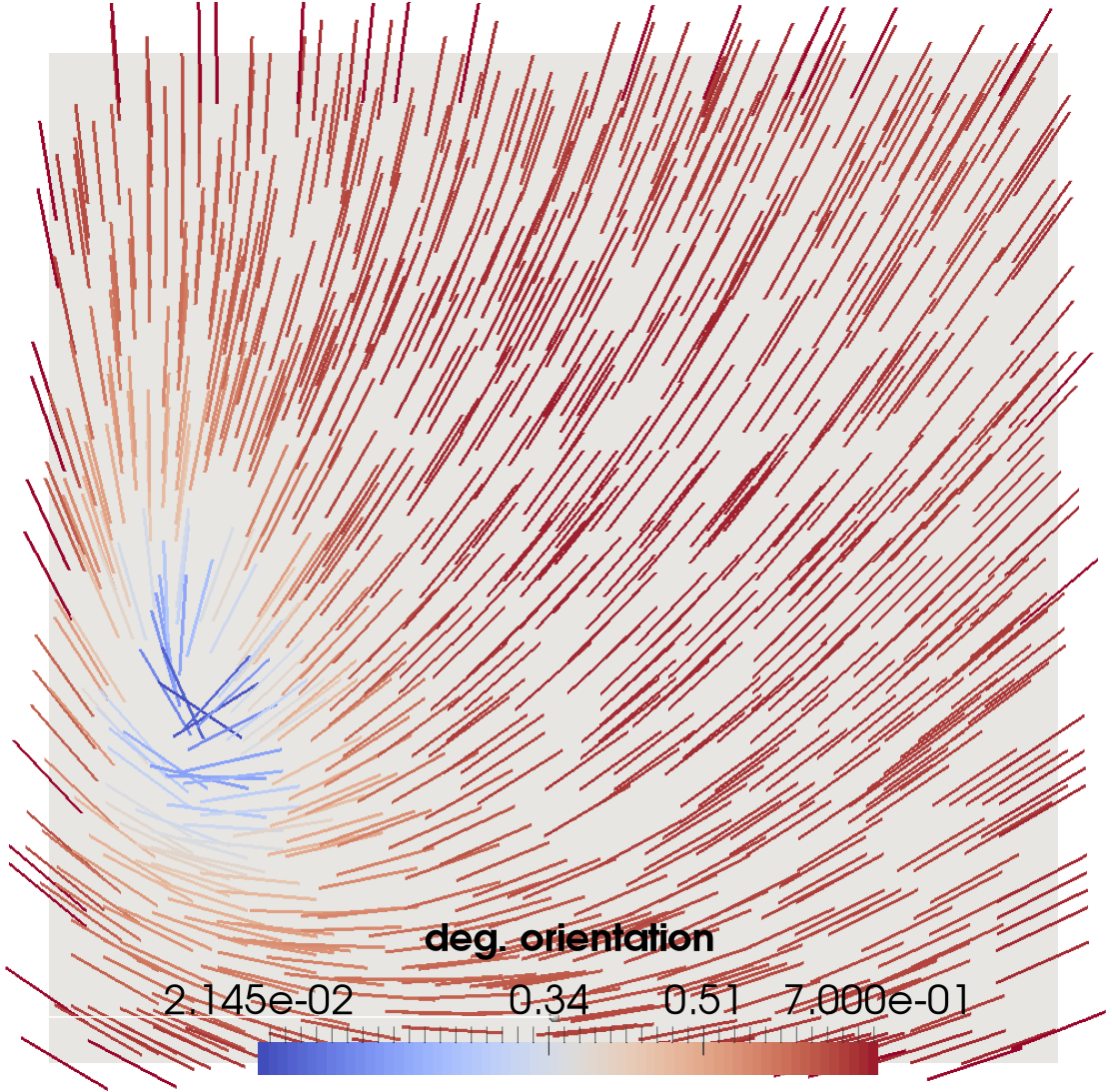} 
\includegraphics[width=0.395\linewidth]{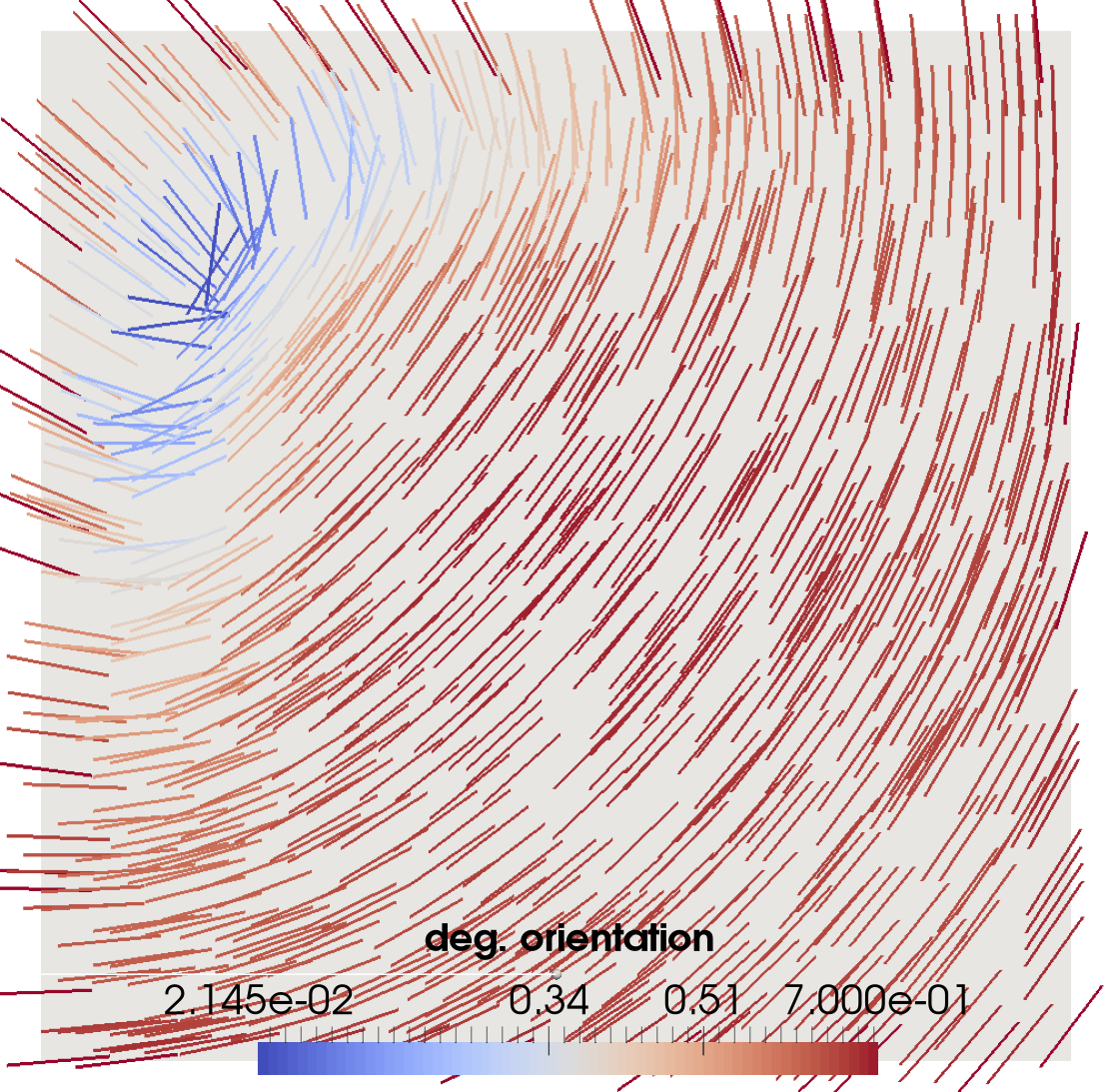} \vspace{0.08in}

\includegraphics[width=0.38\linewidth]{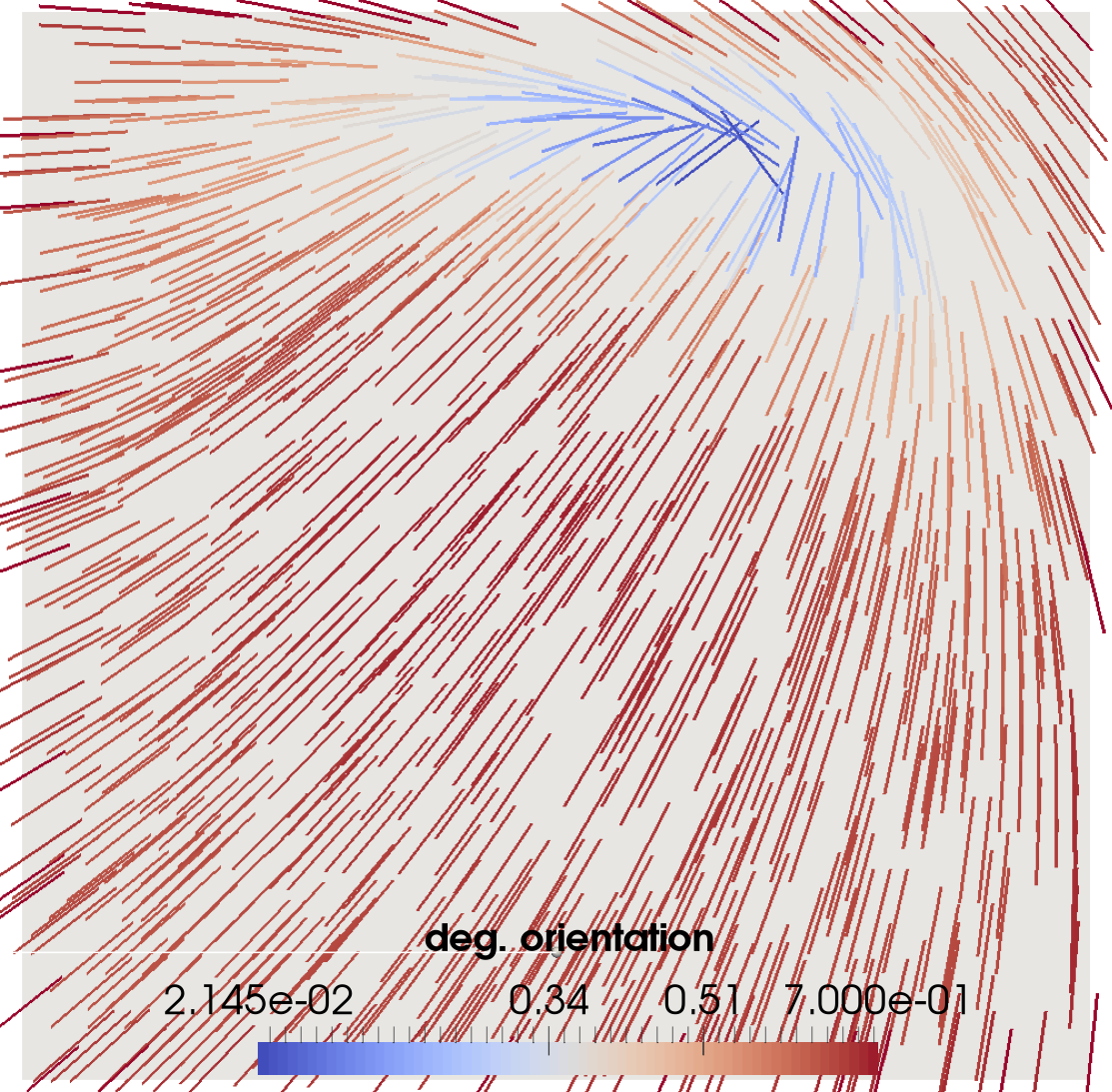} \hspace{0.05in}
\includegraphics[width=0.38\linewidth]{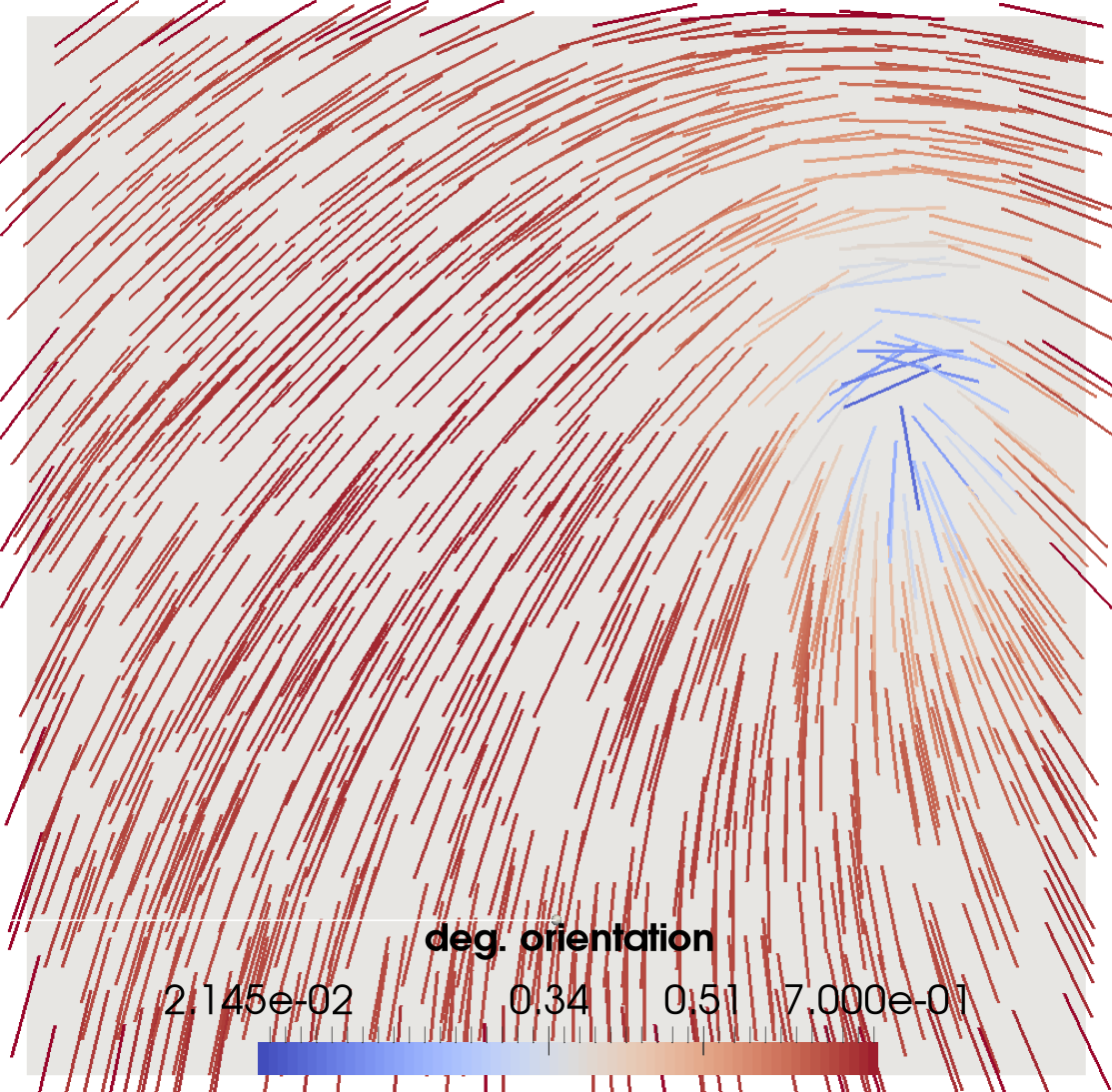} 
\caption{Horizontal slices of the $+1/2$ degree line defect in a 3-D cube domain shown in Figure \ref{fig:half_line_defect_view_A} (Section \ref{sec:line_defect}). Top: left is $z=0.2$, right is $z=0.4$.  Bottom: left is $z=0.6$, right is $z=0.8$.  The location of the point defect in each plane rotates with the boundary conditions.
}
\label{fig:half_line_defect_view_B}
\end{figure}
\end{center}

\subsection{Saturn-ring Defect}\label{sec:saturn_ring_defect}

Next, we simulate the Saturn-ring defect \cite{Alama_PRE2016,Gu_PRL2000} using the double well potential from Section \ref{sec:line_defect} with $\Bulkcoef = 0.09$. 
The domain $\Om$ is a ``prism'' type of cylindrical domain with square cross-section $[-0.25 \sqrt{2}, 0.75 \sqrt{2}]^2$, is centered about the $z=0$ plane, and has height $6$. The domain contains a spherical inclusion, with boundary $\Gm_{i}$, centered at $(\sqrt{2}/4,\sqrt{2}/4,0)$ with radius $0.283/\sqrt{2}$.  See \cite[Sec. 5.1.1]{Nochetto_JCP2018} for a precise definition. 

We use the following Dirichlet boundary conditions on $\Gamma_s = \Gamma_\vNN = \dOm$,
\begin{equation*}
\vn = \vnu, ~ \text{on } \Gm_{i}, \quad
\vn = (0,0,1)\tp, ~ \text{on } \Gm_{o}, \quad \vNN = \vn \otimes \vn, ~ \text{on } \dOm, \quad
s = s^*, ~ \text{on } \dOm,
\end{equation*}
where $\Gm_{o}$ is the outer boundary of $\Om$, $\vnu$ is the outer normal vector of the spherical inclusion,
and $s^*$ is the global minimum of the double well potential $\psi$. The initial conditions in $\Om$ for the gradient flow are: $s = s^*$ and $\vn = (0,0,1)\tp$, which have initial energy $\Euni^h[s_h,\vNN_h] = 7.59906$.

We show the final equilibrium configurations of $s$ and the tensor field $\vNN$ in Figure \ref{fig:saturn_ring_defect}.  A cross-section of the solution is shown that illustrates the $-1/2$ degree nature of the Saturn-ring defect (note: the defect set of a ring about the equator of the inclusion).  The final state has $\Euni^h[s_h,\vNN_h] = 2.98004$ and $\min(s_{h}) \approx 5.026 \times 10^{-2}$. In contrast to our previous experiments using the Ericksen model \cite{Nochetto_JCP2018}, this new simulation is consistent with the physics of liquid crystals \cite{Alama_PRE2016,Gu_PRL2000}.

\begin{center}
\begin{figure}[ht]
\includegraphics[width=0.4\linewidth]{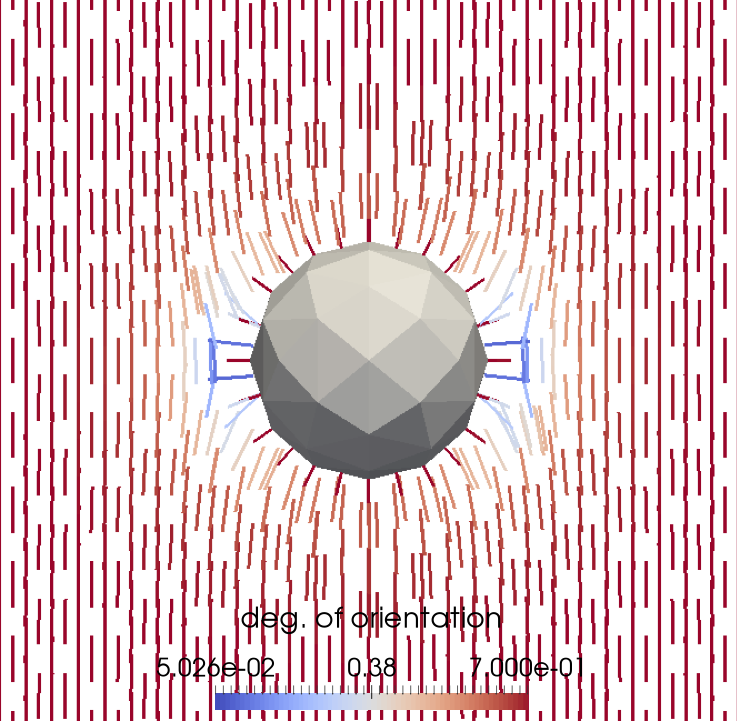} \hspace{0.3cm}
\includegraphics[width=0.4\linewidth]{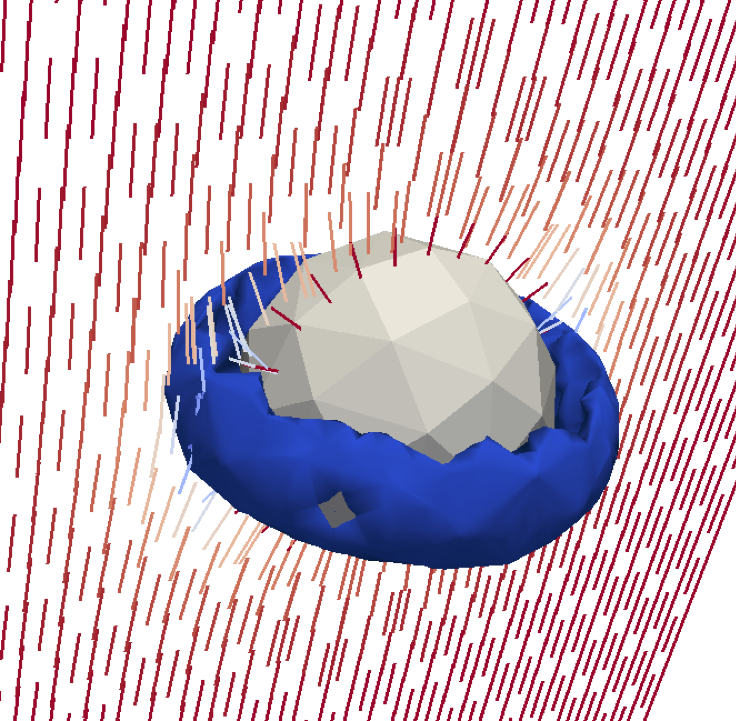}\\
\caption{Saturn-ring defect in 3-D (Section \ref{sec:saturn_ring_defect}). Left: the line field is plotted with color scale based on the degree-of-orientation $s$; the $-1/2$ degree defect is visible on the left and right sides of the spherical inclusion (discretized sphere).  Right: a view of the $s=0.25$ isosurface in blue that contains the ring defect.  The time step for the gradient flow was $\tau = 10^{-3}$. The configuration is symmetric about the vertical axis.  Away from the sphere, the solution is $s = 0.7$ and $\vNN =  (0,0,1) \otimes (0,0,1)$.
}
\label{fig:saturn_ring_defect}
\end{figure}
\end{center}

\section{Conclusions} \label{sec:conclusion}

We introduced a structure-preserving finite element method for a uniaxially-constrained $\vQ$-tensor model of nematic liquid crystals. In such a model, the energy is a degenerate functional of a tensor that must satisfy a rank-one constraint a.e. in the physical domain. We proved the $\Gamma$-convergence of the discrete energies as the mesh size tends to zero and developed an energy-decreasing gradient flow algorithm for the computation of discrete solutions. The numerical experiments show that this method is capable of capturing high-dimensional and non-orientable defect structures.

\subsection*{Acknowledgments}
The authors thank Wenbo Li for pointing out reference \cite{Balan:16} and suggesting an idea for the proof of Lemma \ref{L:Lip-rank1}.


\bibliographystyle{plain}
\bibliography{MasterBibTeX}

\end{document}